\documentclass[3p,final]{article}
\usepackage{graphicx} 
\usepackage{lipsum} 
\usepackage{amsmath}
\usepackage{amsfonts,amssymb}
\usepackage{amsthm}
\usepackage{mathabx}
\usepackage{xcolor}
\usepackage{indentfirst}
\usepackage{enumitem}
\usepackage{hyperref}
\usepackage{tikz}
\usepackage{todonotes}
\usepackage{array}
\usepackage{subcaption}
\usepackage{graphicx}
\usetikzlibrary{positioning, calc}
\usepackage{bbm}

\usepackage[normalem]{ulem}

\usepackage{comment}
\usetikzlibrary{patterns}
\usetikzlibrary{shapes.geometric}
\usepackage{enumitem}

\usetikzlibrary{shapes.geometric, arrows, positioning}
\tikzstyle{circleNode} = [circle, draw, minimum size=1.5em]
\tikzstyle{arrow} = [->, thick]
\allowdisplaybreaks

\newcommand{\E}{\mathbb{E}}
\newcommand{\bbN}{\mathbb{N}}

\newcommand{\calX}{\mathcal{X}}
\newcommand{\calY}{\mathcal{Y}}
\newcommand{\calZ}{\mathcal{Z}}
\newcommand{\ns}{n_\mathcal{S}}
\newcommand{\F}{\mathcal{F}}

\newcommand{\R}{\mathbb{R}}
\newcommand{\Z}{\mathbb{Z}}
\newcommand{\prob}{\mathbb{P}}
\newcommand{\one}{\mathbbm{1}}

\DeclareMathOperator{\Aut}{Aut}
\DeclareMathOperator{\Cov}{Cov}

\newtheorem{theorem}{Theorem}[section]
\newtheorem{definition}[theorem]{Definition}
\newtheorem{proposition}[theorem]{Proposition}
\newtheorem{lemma}[theorem]{Lemma}
\newtheorem{assumption}[theorem]{Assumption}

\title{The Maki--Thompson Model with Spontaneous Stifling on Symmetric Networks}
\author{Nancy L. Garcia, Denis A. Luiz and Daniel M. Machado}
\date{}

\begin{document}
\maketitle
\begin{abstract}
We investigate rumor spreading in a generalized Maki–Thompson model with spontaneous stifling, evolving on quasi-transitive networks. Individuals are either ignorants, spreaders, or stiflers; spreaders stop by contact with other spreaders or stiflers or after an independent random waiting time sampled from a given distribution, modeling a spontaneous loss of interest. The topology of the underlying population network is incorporated by modeling it as a broad class of symmetric networks, whose vertices are partitioned into finitely many orbit types. This yields a unified framework for homogeneous and heterogeneous networks. For sequences of finite quasi-transitive graphs, and for infinite quasi-transitive graphs with subexponential growth, we establish a Functional Law of Large Numbers and a Functional Central Limit Theorem for the densities of each vertex type for the three states. The mean-field  limit is described by a system of nonlinear integral equations, while fluctuations are asymptotically Gaussian and governed by a system of stochastic integral equations with explicit covariance. Our results show how the topology and the law of spontaneous stifling jointly shape the speed and variability of rumor outbreaks. 
 As a special case, our model reduces to the classical Maki–Thompson model when spontaneous stifling is absent.

\end{abstract}



\section{Introduction}

In the last few years, the literature on mathematical models of dissemination of information has increased rapidly as technology provided instantaneous communication \cite{ndii2018mathematical}. Social networks made it possible to propagate information without assuming veracity and it has consequences across diverse domains, including public health, financial markets, and socio-political stability \cite{yang2022health}. As pointed out in \cite{zhang2020health}, the COVID-19 pandemic came with a misinformation pandemic, an \textit{infodemic}, where rumors misled people’s perception and behavior. 

Typically, the literature on mathematical models for rumor spread considers homogeneously mixed populations, that is, the mechanisms for dissemination of information are given in the complete graph \cite{daley1965stochastic,maki1973mathematical,sudbury1985proportion,pittel1990daley}. Recently, the topological features of the underlying network architecture have been incorporated in the models to provide more realistic analyses \cite{nekovee2007theory,moreno2004dynamics,huang2024competitive,bodaghi2019number,li2019dynamical,de2014role,bacca2025optimal}.

A cornerstone of this field is the \textbf{Maki--Thompson model} \cite{maki1973mathematical}, which classifies individuals as \textbf{ignorants}, \textbf{spreaders}, or \textbf{stiflers}. In this classic setup, ignorant individuals become spreaders upon contact with one, and later turn into stiflers after contacting someone who is already aware of the rumor, thereby ceasing to spread it. This dynamics is driven by a continuous--time Markov chain (CTMC) in a subset of $\Z^3$, namely in $\Omega^{(n)}=\{(z_1,z_2,z_3)\in\Z^3:z_1,z_2,z_3\geq0,z_1+z_2+z_3=n\}$, where the coordinates denote the number of individuals in each class: ignorants, spreaders and stiflers, respectively. The transitions and probabilities of the CTMC in an interval $(t,t+\Delta t)$ are given by
\begin{equation}\label{DKtransition}
\begin{array}{ccc}
\text{Transition} \quad &\text{Probability} \\[0.2cm]
(z_1,z_2,z_3)\to(z_1-1,z_2+1,z_3) &z_1z_2 \Delta t+o(\Delta t),\\
(z_1,z_2,z_3)\to(z_1,z_2-1,z_3+1) &(n-z_1)z_2\Delta t+o(\Delta t).
\end{array}
\end{equation}

In this paper, we analyze a natural generalization of the classical Maki--Thompson model: the \textbf{Maki--Thompson model with spontaneous stifling}. We introduce a mechanism where spreaders may become stiflers not only through contact but also after a random waiting time, modeling a loss of interest or motivation. This creates a more realistic dynamic, reflecting a competition between contact-driven propagation and time-driven deactivation.

The main contribution of our analysis is the introduction of topological characteristics of the underlying network of the population, which is set on a broad class of \textbf{symmetric networks} known as \textbf{quasi-transitive graphs}. These are graphs whose vertices can be partitioned into a finite number of classes, or types, where all vertices of the same type are structurally equivalent under graph automorphisms. This framework is general enough to encompass a wide variety of networks — from regular lattices to more complex structures with designed heterogeneity — while retaining enough regularity to be analytically tractable. For types $i$ and $j$, the key structural parameters are the inter-type neighbor counts, denoted by $n_i(j)$, and the proportion of vertices of each type, $p_j$.

We establish two primary results that characterize the macroscopic behavior of the rumor process in these networks: the Functional Law of Large Numbers (FLLN) and the Functional Central Limit Theorem (FCLT).

The paper is structured as follows. In Section 2, we introduce the necessary graph-theoretic preliminaries, including the definition of quasi-transitive graphs and the key properties used throughout the paper. In Section 3, we give a formal description of the Maki--Thompson model with spontaneous stifling and state our main results. Section 4 contains the proofs of our results. In Section 5, we present simulations that exemplify the effects of the different types of graphs, the stifling-time law and different types of vertices.

\section{Graph-Theoretic Preliminaries}
\label{sec:graphs}

To analyze rumor dynamics on networks, we require a framework that captures both symmetry and controlled heterogeneity. In this section, we introduce the terminology and properties of \emph{quasi-transitive graphs}, which provide the structural setting for our results. These graphs strike a balance: they are general enough to encompass a broad range of networks, from regular lattices to structured heterogeneous systems, yet symmetric enough to admit tractable analysis.

\subsection*{Quasi-Transitive Graphs and Their Structure}

Let $\Gamma=(V,E)$ be a connected, locally finite graph (i.e., $\deg(v)<\infty$ for all $v\in V$). A \emph{path of length $m\in\bbN\cup\{0\}$} is a finite sequence of pairwise distinct vertices $(v_0,v_1,\ldots,v_m)$ such that $v_{i-1}\sim v_i$ for every $i=1,\ldots,m$. The vertices $v_0$ and $v_m$ are the \emph{endpoints}. The length of the path is the number of edges $m$.

A bijection $\varphi:V_{1}\to V_{2}$ is an \textit{isomorphism} between graphs $\Gamma_{1}=(V_{1},E_{1})$ and $\Gamma_{2}=(V_{2},E_{2})$ if
\[
	u\sim v \iff \varphi(u)\sim \varphi(v).
\]
If such a map exists, the graphs are \emph{isomorphic}, written $\Gamma_{1}\cong\Gamma_{2}$. An \emph{automorphism} is an isomorphism from $\Gamma$ to itself; the group of all automorphisms is denoted $\Aut(\Gamma)$.

Define the (graph) distance $d:V\times V\to\bbN\cup\{0\}$ by letting $d(v,w)$ be the length of a shortest path joining $v$ and $w$ (when it exists). For $r\in\bbN$ and $v\in V$, write
\[
	B_{r}[v]:=\{w\in V:\ d(w,v)\le r\}, \qquad \beta_{v}(r):=\#(B_{r}[v]).
\]
We let $G=\Aut(\Gamma)$ act on $V$ by graph automorphisms.

The action of $G$ on $V$ is \emph{quasi-transitive} if there are finitely many vertex orbits, i.e., $V/G$ is finite. Equivalently, there exist representatives $v_{1},\dots,v_{\ns}\in V$ such that
\[
	V=\bigsqcup_{i=1}^{\ns}G\cdot v_{i}.
\]
A graph $\Gamma$ is \emph{quasi-transitive} if the action of $G$ on $V$ is quasi-transitive. The special case $\ns=1$ is \emph{vertex-transitive}.

Let $\tau:V\to\{1,\dots,\ns\}$ be the \emph{type map} defined by $\tau(v)=i$ when $v\in G\cdot v_{i}$. Write $V_{i}:=G\cdot v_{i}$ for the type-$i$ vertex class. For $i,j\in\{1,\dots,\ns\}$, set
\[
	n_{i}(j):=\#\{\,w\in V_{j}:\ w\sim v\,\} \quad\text{for any }v\in V_{i};
\]
by quasi-transitivity, $n_{i}(j)$ is well defined (independent of $v\in V_{i}$).
These inter-type neighbor counts $n_i(j)$, together with the type proportions introduced below, will be the key structural parameters governing the rumor dynamics.

\subsection*{Quasi-Transitive Sequences}

Our first main result is established in the setting of an increasing sequence of finite graphs with stable local structure.

\begin{definition}
	\label{def:QTS}
	Let $(\Gamma^{n})_{n\in\bbN}$ be a sequence of quasi-transitive graphs $\Gamma^{n}=(V^{n},E^{n})$ that is \emph{uniformly locally finite}, i.e., there exists $d\in\bbN$ such that $\deg(v)\le d$ for all $v\in V^{n}$ and all $n$. Assume:
	\begin{enumerate}
		\item Each $\Gamma^{n}$ has the same number $\ns$ of types.
		\item For all $i,j$, the inter-type neighbor counts $n_{i}(j)$ are independent of $n$.
		\item $|V^{n}|\to\infty$ as $n\to\infty$.
	\end{enumerate}
	Then $(\Gamma^{n})_{n\in\bbN}$ is called a Quasi-Transitive Sequence (QTS).
\end{definition}

For each type \(i\), write
\[
V_i^n:=\{v\in V^n:\tau(v)=i\}, \qquad p_i^n:=\frac{|V_i^n|}{|V^n|}.
\]

To illustrate Definition~\ref{def:QTS}, we present some examples of QTS
\((\Gamma^n)_{n\in\bbN}\) together with their corresponding finite type set
\(\mathcal T=\{1,\dots,\ns\}\) and inter-type neighbor counts \(n_i(j)\).

\paragraph{Example 1 (\(\Gamma_n^{2}\), cycle family)}
Consider the graph $\Gamma_n^2$, which has $n$ vertices and degree $2$ at every vertex. In Figures \ref{fig:g5-7} and \ref{fig:g5-13}  we provide two instances. 
Here \(\ns=1\) (a single vertex type), $n_{1}(1)=2$, and  $p_1^n=1$.
Thus \((\Gamma_n^{2})_{n\in\bbN}\) is a transitive QTS.

\paragraph{Example 2 (\(\Gamma_n^{2,4}\), two-type bipartite pattern)}

Figures \ref{fig:g24-8} and \ref{fig:g24-13} present examples of graphs we denote by $\Gamma_n^{2,4}$, where $n$ is the number of degree-$4$ vertices. In these graphs, there are only degree-$2$ and degree-$4$ vertices with the following connection pattern: every degree-$2$ vertex (type $1$) connects to $2$ degree-$4$ vertices (type $2$), and every degree-$4$ vertex connects to $4$ degree-$2$ vertices.

Let type \(1\) be the degree-\(2\) vertices and type \(2\) the degree-\(4\) vertices.
The inter-type neighbor counts are
\[
n_{1}(1)=0,\quad n_{1}(2)=2,\qquad
n_{2}(2)=0,\quad n_{2}(1)=4,
\]
independently of \(n\). Hence \((\Gamma_n^{2,4})_{n\in\bbN}\) is a QTS. 
Since $\#V_1^n = 2n$ and $\#V_2^n = n$ by construction, the type proportions are constant:
$p_1^n = 2/3$ and $p_2^n = 1/3$.

\begin{figure}[h]
    \centering
    \begin{subfigure}{0.2\textwidth}
        \centering
        \begin{tikzpicture}[every node/.style={circle, draw, fill=black, inner sep=1.5pt},scale=0.6]
            \foreach \i in {1,...,7} {
                \node (v\i) at (90-360/7*\i:2cm) {};
            }
            \foreach \i [evaluate={\j=int(mod(\i,7)+1)}] in {1,...,7} {
                \draw (v\i) -- (v\j);
            }
        \end{tikzpicture}
        \caption{$\Gamma_7^{2}$}
        \label{fig:g5-7}
    \end{subfigure}
    \begin{subfigure}{0.23\textwidth}
        \centering
        \begin{tikzpicture}[every node/.style={circle, draw, fill=black, inner sep=1.5pt},scale=0.6]
            \foreach \i in {1,...,13} {
                \node (v\i) at (90-360/13*\i:2cm) {};
            }
            \foreach \i [evaluate={\j=int(mod(\i,13)+1)}] in {1,...,13} {
                \draw (v\i) -- (v\j);
            }
        \end{tikzpicture}
        \caption{$\Gamma_{13}^{2}$}
        \label{fig:g5-13}
    \end{subfigure}
    \hfill
    \begin{subfigure}{0.22\textwidth}
        \centering
        \begin{tikzpicture}[
            deg4/.style={circle, draw, inner sep=2pt, minimum size=4pt},
            deg2/.style={circle, draw, fill=black, inner sep=1.5pt},
            scale=1
        ]
            \foreach \i in {1,...,8} {
                \node[deg4] (v\i) at (112.5-45*\i:1.2cm) {};
            }
            \foreach \i in {1,...,8} {
                \node[deg2] (u\i) at (90-45*\i:1.3cm) {};  
                \node[deg2] (w\i) at (90-45*\i:1cm) {};   
            }
            \foreach \i [evaluate={\j=int(mod(\i,8)+1)}] in {1,...,8} {
                \draw (v\i) -- (u\i) -- (v\j) -- (w\i) -- (v\i);
                \draw (v\j) -- (w\i);
            }
        \end{tikzpicture}
        \caption{$\Gamma_8^{2,4}$}
        \label{fig:g24-8}
    \end{subfigure}
    \begin{subfigure}{0.23\textwidth}
        \centering
        \begin{tikzpicture}[
            deg4/.style={circle, draw, inner sep=2pt, minimum size=4pt},
            deg2/.style={circle, draw, fill=black, inner sep=1.5pt},
            scale=1
        ]
            \foreach \i in {1,...,13} {
                \node[deg4] (v\i) at (106-360/13*\i:1.2cm) {};
            }
            \foreach \i in {1,...,13} {
                \node[deg2] (u\i) at (90-360/13*\i:1.3cm) {};  
                \node[deg2] (w\i) at (90-360/13*\i:1cm) {};   
            }
            \foreach \i [evaluate={\j=int(mod(\i,13)+1)}] in {1,...,13} {
                \draw (v\i) -- (u\i) -- (v\j) -- (w\i) -- (v\i);
                \draw (v\j) -- (w\i);
            }
        \end{tikzpicture}
        \caption{$\Gamma_{13}^{2,4}$}
        \label{fig:g24-13}
    \end{subfigure}

    \caption{Graphs $\Gamma_7^{2}$, $\Gamma_{13}^{2}$, $\Gamma_8^{2,4}$ and $\Gamma_{13}^{2,4}$.}
    \label{fig:all-gammas}
\end{figure}
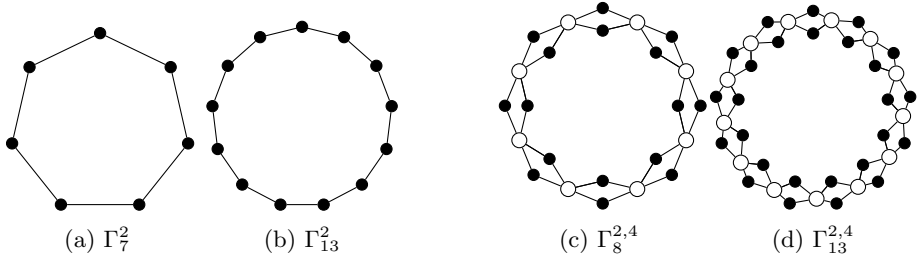

\paragraph{Example 3 (Decorated grid, four–type QTS with mixed degrees)}
Let $\Gamma_{m,n}$ be the graph with vertex set $V=\mathbb{Z}_{m}\times\mathbb{Z}_{n}$, where $m$ and $n$ are even numbers, and edges
$(x,y)\sim(x\pm1,y)$ and $(x,y)\sim(x,y\pm1)$ (indices taken modulo $m,n$). 
Partition $V$ into four types by parity,
\begin{align*}
A&=\{(x,y):x\equiv0\!\!\!\pmod2,\ y\equiv0\!\!\!\pmod2\},\quad
B=\{(x,y):x\equiv0\!\!\!\pmod2,\ y\equiv1\!\!\!\pmod2\},\\
C&=\{(x,y):x\equiv1\!\!\!\pmod2,\ y\equiv0\!\!\!\pmod2\},\quad
D=\{(x,y):x\equiv1\!\!\!\pmod2,\ y\equiv1\!\!\!\pmod2\}.
\end{align*}
Add the following \emph{decorations}, independent of $m,n$:
\begin{itemize}
  \item for each $A$–vertex $(x,y)$, add an edge to $(x+1,y+1)$ (a $D$–vertex);
  \item for each $C$–vertex $(x,y)$, add edges to $(x\pm2,y)$ (both $C$–vertices).
\end{itemize}
Then $\ns=4$ and the inter–type neighbor counts $n_i(j)$ (rows indexed by source type $i\in\{A,B,C,D\}$) are
\[
N=\bigl(n_i(j)\bigr)=
\begin{pmatrix}
0&2&2&1\\
2&0&0&2\\
2&0&2&2\\
1&2&2&0
\end{pmatrix},
\quad\text{so}\quad
\deg(A)=5,\ \deg(B)=4,\ \deg(C)=6,\ \deg(D)=5.
\]

As $m,n$ are even, the four types partition the grid evenly, yielding constant proportions
$p_A^n = p_B^n = p_C^n = p_D^n = 1/4$.

\begin{figure}[h]
\centering
\begin{tikzpicture}[
  scale=0.9,
  A/.style={circle,    draw, fill=white, inner sep=1.2pt},  
  B/.style={circle,    draw, fill=black, inner sep=1.2pt},  
  C/.style={rectangle, draw, fill=white, inner sep=0pt, minimum size=3pt}, 
  D/.style={rectangle, draw, fill=black, inner sep=0pt, minimum size=3pt}, 
  gridedge/.style={line width=0.35pt},
  ADedge/.style={line width=0.8pt, dashed},
  CCedge/.style={line width=1pt, dash pattern=on 3pt off 1.5pt} 
]
  \def\M{8}
  \def\N{6}
  
  \foreach \x in {0,...,\numexpr\M-1} {
    \foreach \y in {0,...,\numexpr\N-1} {
      \ifodd\x
        \ifodd\y \node[D] (v\x\y) at (\x,\y) {};
        \else     \node[C] (v\x\y) at (\x,\y) {};
        \fi
      \else
        \ifodd\y \node[B] (v\x\y) at (\x,\y) {};
        \else     \node[A] (v\x\y) at (\x,\y) {};
        \fi
      \fi
    }
  }

  \foreach \x in {0,...,\numexpr\M-1} {
    \foreach \y in {0,...,\numexpr\N-1} {
      \pgfmathtruncatemacro{\xp}{Mod(\x+1,\M)}
      \pgfmathtruncatemacro{\yp}{Mod(\y+1,\N)}
      \draw[gridedge] (v\x\y) -- (v\xp\y);
      \draw[gridedge] (v\x\y) -- (v\x\yp);
    }
  }

  \foreach \x in {0,...,\numexpr\M-1} {
    \foreach \y in {0,...,\numexpr\N-1} {
      
      \ifodd\x\relax\else
        \ifodd\y\relax\else
          \pgfmathtruncatemacro{\xad}{Mod(\x+1,\M)}
          \pgfmathtruncatemacro{\yad}{Mod(\y+1,\N)}
          \draw[ADedge] (v\x\y) -- (v\xad\yad);
        \fi
      \fi
    }
  }

  \foreach \x in {0,...,\numexpr\M-1} {
    \foreach \y in {0,...,\numexpr\N-1} {
      
      \ifodd\x
        \ifodd\y\relax\else
          \pgfmathtruncatemacro{\xpp}{Mod(\x+2,\M)}
          \pgfmathtruncatemacro{\xmm}{Mod(\x-2,\M)}
          
          \draw[CCedge] (v\x\y) to[out=15,in=165] (v\xpp\y);
          
          \draw[CCedge] (v\x\y) to[out=-15,in=-165] (v\xmm\y);
        \fi
      \fi
    }
  }

  \begin{scope}[shift={(0,-1.4)}]
    \node[A] at (0,0) {}; \node at (0.45,0) {$A$ \; (deg $5$)};
    \node[B] at (2.2,0) {}; \node at (2.65,0) {$B$ \; (deg $4$)};
    \node[C] at (4.4,0) {}; \node at (4.85,0) {$C$ \; (deg $6$)};
    \node[D] at (6.6,0) {}; \node at (7.05,0) {$D$\;  (deg $5$)};
  \end{scope}

\end{tikzpicture}7
\caption{Decorated grid: four-type QTS with degrees: $B=4$, $A=5$, $D=5$, $C=6$. Note that $A$ and $D$ share the same degree but  have different types, as their neighbor sets have different type compositions}
\end{figure}

\subsection{Graphs with Subexponential Growth}

Our second main result is formulated for infinite graphs, where a notion of growth rate becomes essential. Two functions $f,g:\bbN\to(0,\infty)$ are \emph{coarsely equivalent}, written $f\asymp g$, if there exists $C\ge1$ such that
\[
	C^{-1}\, g(\lfloor r/C\rfloor)\ \le\ f(r)\ \le\ C\, g(\lceil Cr\rceil)\qquad
	\text{for all }r\in\bbN.
\]
The \emph{growth type} of $f$ is the equivalence class
\[
	[f]\;:=\;\{\,g:\bbN\to(0,\infty)\ :\ g\asymp f\,\}.
\]

For quasi-transitive graphs, the ball volume $\beta_{v}(r)$ is constant on each orbit class. Moreover, the growth type is independent of the basepoint (\cite{Woess2000}, Prop.~3.9 and Lemma~3.13). The \emph{growth of $\Gamma$}, denoted by $\beta(r)$, is the growth type of $\beta_{v}(r)$.

The regime of interest in this paper is that of \emph{subexponential growth}. An infinite quasi-transitive graph $\Gamma$ has subexponential growth if
\[
	\lim_{r\to\infty}\frac{1}{r}\log \beta(r)=0.
\]
This condition ensures that finite-volume approximations of the infinite graph are sufficiently well behaved for our probabilistic analysis. In particular, it allows us to control boundary effects, which play a central role in Section~5.

\subsubsection*{Examples of graphs with subexponential growth.}
All examples below are connected, locally finite, and have subexponential growth (indeed polynomial, except where explicitly noted), while illustrating the distinction between $\ns=1$ (transitive) and $\ns>1$ (quasi-transitive) cases.

\textbf{Transitive examples.}
The $d$-dimensional integer lattice $\mathbb{Z}^{d}$ has vertex set $V=\mathbb{Z}^{d}$ with edges between vertices at Euclidean distance $1$; it has polynomial growth of degree $d$ (i.e., $\#B_{r}\asymp r^{d}$). Likewise, graphs arising from \emph{regular tessellations} of the Euclidean plane are transitive with quadratic growth ($\#B_{r}\asymp r^{2}$). \emph{Cayley graphs} are vertex-transitive by construction. For instance, the Cayley graph of the \emph{discrete Heisenberg group $H_{3}(\mathbb{Z})$} (upper-triangular $3\times3$ integer matrices with ones on the diagonal) has polynomial growth of degree $4$ ($\#B_{r}\asymp r^{4}$). Beyond the polynomial case, the \emph{Grigorchuk group} provides a classical example of a group with \emph{intermediate growth}, meaning its growth function is subexponential yet faster than any polynomial.

\textbf{Quasi-transitive but non-transitive examples.}
\emph{Comb graph.} Here $V=\mathbb{Z}\times\{0,1\}$ with edges $(n,0)\sim(n+1,0)$ and $(n,0)\sim(n,1)$. Growth is linear ($\#B_{r}\asymp r$). There are $\ns=2$ types: backbone vertices (degree $3$) and leaf vertices (degree $1$).

\emph{Infinite strip $\mathbb{Z}\square P_{3}$.} This is the Cartesian product of the line with a 3-vertex path. Growth is again linear ($\#B_{r}\asymp r$). There are $\ns=2$ types: the middle row (degree $4$) and the boundary rows (degree $3$).

\emph{Graphs from $k$-uniform tilings.} A tiling is called $k$-uniform if its vertices form exactly $k$ distinct orbits under the tiling's symmetry group. The associated graphs have quadratic growth ($\#B_{r}\asymp r^{2}$). While $1$-uniform tilings (which include all regular and Archimedean tilings) are vertex-transitive, tilings with $k>1$ are, by definition, quasi-transitive but not transitive. In Figure~\ref{fig:4-uniform}, we depict a structure with four distinct vertex types ($\ns=4$).

\emph{Dual of the trihexagonal tiling.} Vertices correspond to faces (triangles or hexagons) of the $3.6.3.6$ tiling, with adjacencies whenever two faces touch. Growth is quadratic ($\#B_{r}\asymp r^{2}$). There are $\ns=2$ types (degrees $3$ and $6$), so the graph is not vertex-transitive but is quasi-transitive.

\begin{figure}[h]
\centering

\begin{subfigure}[t]{0.28\textwidth}
\centering
\begin{tikzpicture}[
    scale=0.7,
    vertex/.style={circle,fill=black,inner sep=0.8pt}
  ]
  \foreach \x in {-3,...,3}{
    \foreach \y in {-3,...,3}{
      \ifnum\x<3 \draw[gray!60] (\x,\y)--(\x+1,\y);\fi
      \ifnum\y<3 \draw[gray!60] (\x,\y)--(\x,\y+1);\fi
    }
  }
  \foreach \x in {-3,...,3}{
    \foreach \y in {-3,...,3}{
      \node[vertex] at (\x,\y) {};
    }
  }
\end{tikzpicture}
\caption{Square lattice $\mathbb Z^2$.}
\label{fig:lattice}
\end{subfigure}
\hfill
\begin{subfigure}[t]{0.34\textwidth}
\centering
\begin{tikzpicture}[
    scale=0.88,
    hex/.style={circle,fill=black!70,inner sep=0.8pt},
    tri/.style={circle,fill=black!100,inner sep=0.8pt}
  ]
  \def\rt{0.8660254} 
  \foreach \i in {-1,...,4}{
    \foreach \j in {-1,...,4}{
      \pgfmathsetmacro\x{\i + 0.5*\j}
      \pgfmathsetmacro\y{\rt*\j}
      \coordinate (H-\i-\j) at (\x,\y);
      \node[hex] at (H-\i-\j) {};
    }
  }
  \foreach \i in {-1,...,3}{
    \foreach \j in {-1,...,3}{
      \pgfmathsetmacro\xA{\i + 0.5*\j}
      \pgfmathsetmacro\yA{\rt*\j}
      \pgfmathsetmacro\xB{\i+1 + 0.5*\j}
      \pgfmathsetmacro\yB{\rt*\j}
      \pgfmathsetmacro\xC{\i + 0.5*(\j+1)}
      \pgfmathsetmacro\yC{\rt*(\j+1)}
      \pgfmathsetmacro\xT{(\xA+\xB+\xC)/3}
      \pgfmathsetmacro\yT{(\yA+\yB+\yC)/3}
      \coordinate (Tu-\i-\j) at (\xT,\yT);
      \draw[gray!60] (Tu-\i-\j) -- (H-\i-\j);
      \draw[gray!60] (Tu-\i-\j) -- (H-\the\numexpr\i+1\relax-\j);
      \draw[gray!60] (Tu-\i-\j) -- (H-\i-\the\numexpr\j+1\relax);
      \node[tri] at (Tu-\i-\j) {};
    }
  }
  \foreach \i in {-1,...,3}{
    \foreach \j in {-1,...,3}{
      \pgfmathsetmacro\xA{\i+1 + 0.5*\j}
      \pgfmathsetmacro\yA{\rt*\j}
      \pgfmathsetmacro\xB{\i+1 + 0.5*(\j+1)}
      \pgfmathsetmacro\yB{\rt*(\j+1)}
      \pgfmathsetmacro\xC{\i + 0.5*(\j+1)}
      \pgfmathsetmacro\yC{\rt*(\j+1)}
      \pgfmathsetmacro\xT{(\xA+\xB+\xC)/3}
      \pgfmathsetmacro\yT{(\yA+\yB+\yC)/3}
      \coordinate (Td-\i-\j) at (\xT,\yT);
      \draw[gray!60] (Td-\i-\j) -- (H-\the\numexpr\i+1\relax-\j);
      \draw[gray!60] (Td-\i-\j) -- (H-\the\numexpr\i+1\relax-\the\numexpr\j+1\relax);
      \draw[gray!60] (Td-\i-\j) -- (H-\i-\the\numexpr\j+1\relax);
      \node[tri] at (Td-\i-\j) {};
    }
  }
\end{tikzpicture}
\caption{Dual graph of the trihexagonal tiling.}
\label{fig:trihex-dual}
\end{subfigure}
\hfill
\begin{subfigure}[t]{0.30\textwidth}
\centering
\input{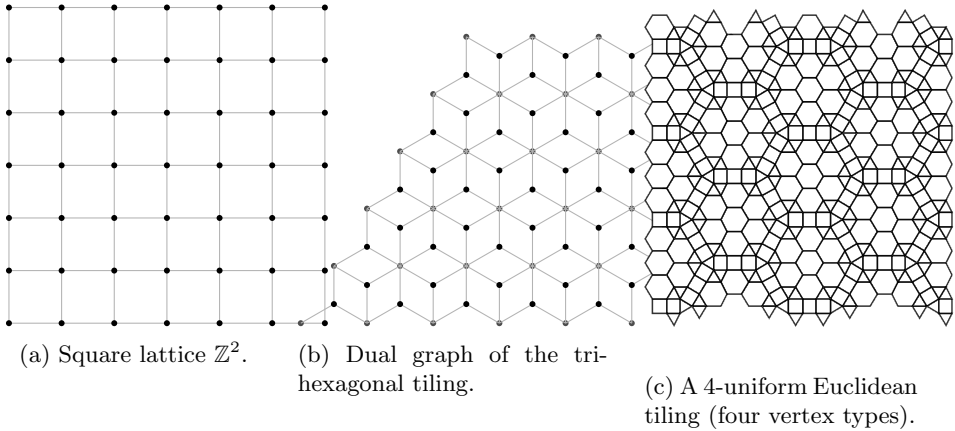}
\caption{A 4-uniform Euclidean tiling (four vertex types).}
\label{fig:4-uniform}
\end{subfigure}

\caption{Planar graphs with quadratic growth: the square lattice, the dual of the trihexagonal tiling, and a 4-uniform tiling.}
\label{fig:quadratic-growth-examples}
\end{figure}

\section{The model}

We consider the Maki--Thompson model with spontaneous stifling on quasi-transitive graphs. As in the classical Maki--Thompson model, a contact between an ignorant and a spreader turns the ignorant into a spreader, whereas a contact  between a spreader and a non-ignorant makes the spreader to become a stifler. We introduce another feature to the dynamics: the possibility of a spreader individual to become a stifler by waiting a random time. Observe that, if we set this time to be infinite almost surely, then we recover the Maki--Thompson model. 

Let $(\Gamma^{n})_{n\in\bbN}=(V^n,E^n)_{n\in\bbN}$ be a sequence of graphs with $\ns$ types for any $n$. For any edge $\{u,v\}\in E^{n}$, let $N_{u,v}$ be a Poisson process of rate $\lambda$ with marks $\theta_{u,v}^{i}$ for $i\geq1$:
\begin{equation}
	N_{u,v}(t):=\sum_{i=1}^{\infty}\one_{\{t\leq\theta_{u,v}^{i}\}}.
\end{equation}

Assume $N_{u_1,v_1}$ independent of $N_{u_2,v_2}$ whenever $(u_{1},v_{1})\neq(u_{2},v_{2})$. Let $\calX^{n}(t),\calY^{n}(t)$ and $\calZ^{n}(t)$ be the sets of vertices in $V^{n}$ at time $t$ in the ignorant, spreader, and stifler state, respectively. Let $\eta$ be an $F$-distributed random variable and for any $k\in\{1,\dots,\ns\}$, let $(\eta^{k,0}_{i})_{i\geq1}$ and $(\eta_{i}^{k})_{i\geq1}$ be i.i.d. copies of $\eta$, independent of each other. We represent the dynamics of the model in Figure \ref{fig:dyn}.

\begin{figure}[ht]
    \centering
    \begin{tikzpicture}[node distance=2.5cm, every node/.style={circle, minimum size=1.5em, align=center},
    label style/.style={inner sep=1pt, font=\normalsize} ]
    \node[draw, minimum height=1.5em, minimum width=3em] (Xn) {$\calX_k^n$};
    \node[draw, minimum height=1.5em, minimum width=3em, right=of Xn] (Yn) {$\calY_k^n$};
    \node[draw, minimum height=1.5em, minimum width=3em, right=of Yn] (Zn) {$\calZ_k^n$};

    \draw[->] (Xn) -- (Yn) node[midway, above, inner sep=0pt] {$\calX_k\leftrightharpoons \calY$};
    \draw[double, -implies] (Yn) -- (Zn); 
    \node[label style] at ($(Yn)!0.5!(Zn)+(0, 0.7cm)$) {$\calY_k\leftrightharpoons(\calY\cup\calZ)$};
    \draw[->] (Yn) to [out=330,in=210] (Zn); 
    \node[label style] at ($(Yn)!0.5!(Zn)+(0, -1cm)$) {$\eta\sim F$};
\end{tikzpicture}
    \caption{Diagram representation of the model. Here, $\calX^n_k=\calX^n\cap V^n_k,\calY^n_k=\calY^n\cap V^n_k$ and $\calZ^n_k=\calZ^n\cap V^n_k$. The arrows above indicate transitions between two states with their corresponding rates above and the arrow below indicate the spontaneous transition with a time that is an independent copy of the $F$-distributed random variable $\eta$. The symbol $\leftrightharpoons$ represents contact between the two sets. }\label{fig:dyn}
\end{figure}
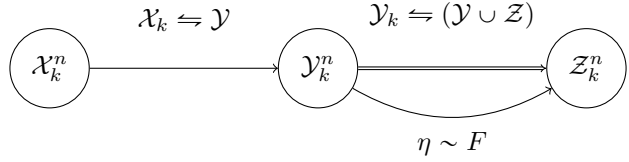

Set
\begin{equation}
	X^{n}_{k}(t):=\#X^{n}_{k}(t), \quad Y^{n}_{k}(t):=\#Y^{n}_{k}(t),\quad Z^{n}_{k}(t):=\#Z^{n}_{k}(t)
\end{equation}
for $k\in\{1,\dots,\ns\}$ and let $A_{k}^{n}$ be the process that counts the interactions ignorant--spreader when the ignorant is in $V_{k}^{n}$ and $B_{k}^{n}$ counts the interactions spreader--nonignorant when the spreader is in $V_{k}^{n}$. That is,

\begin{equation}
	A_{k}^{n}(t)=\sum_{v\in V_k}\sum_{y:y\sim v}\sum_{i=1}^{N_{v,y}(t)}\one_{\{v\in
		\calX^{n}(\theta_{v,y}^{i}-)\}}\one_{\{y\in \calY^{n}(\theta_{v,y}^{i}-)\}}
\end{equation}
and
\begin{equation}
	B_{k}^{n}(t)=\sum_{v\in V_k}\sum_{z:z\sim v}\sum_{i=1}^{N_{v,y}(t)}\one_{\{v\in
		\calY^{n}(\theta_{v,z}^{i}-)\}}\one_{\{z\notin \calX^{n}(\theta_{v,z}^{i}-)\}}
	.
\end{equation}

The process is driven by the following system:
\begin{eqnarray*}
	X^n_k(t)&=&\#V_k^n-Y^n_k(t)-Z^n_k(t),\\
	Y^n_k(t)&=&\sum_{i=1}^{Y^n_k(0)}\one_{\{t<\eta_i^{k,0}\}} +\sum_{j=1}^{A^{n,k}(t)}\one_{\{t<\tau_{j}^{n,k}+\eta_j^k\}}
	-B^{n,k}(t),\mbox{ and}\\
	Z^n_k(t)&=&Z^n_k(0)+\sum_{i=1}^{Y^n_k(0)}\one_{\{t\geq\eta_i^{k,0}\}}+\sum_{j=1}^{A^{n,k}(t)}\one_{\{t\geq\tau_{j}^{n,k}+\eta_j^k\}}
	+B^{n,k}(t),
\end{eqnarray*}
where $\tau_{i}^{n,k}$ is the $i$-th epoch of the counting process $A^{n,k}$.

In order to establish our results, we need two conditions: the first one is topological, while the second concerns the convergence of the initial proportions.

\begin{assumption}[Topological Growth Condition]\label{def:TGC}
    Let $(\Gamma_n)_{n\in\bbN}$ be a sequence of graphs with $\ns$ types. Assume that either one of the following holds:
    \begin{enumerate}
        \item The sequence $(\Gamma_{n})_{n\in\bbN}$ satisfies the QTS property (Definition \ref{def:QTS});
        \item There is a rooted infinite quasi-transitive graph of subexponential growing, $\Gamma$, such that $\Gamma_n\sim B(o,n)\subset\Gamma$.
    \end{enumerate}
\end{assumption}

\begin{assumption}
	\label{initialconv} For any $k\in\{1,\dots,\ns\}$, it holds as $n\to\infty$,
	\begin{equation*}
		\frac{1}{\#V^{n}}(X_{k}^{n}(0),Y_{k}^{n}(0),Z_{k}^{n}(0))\to(\bar{X}_{k}
			(0),\bar{Y}_{k}(0),\bar{Z}_{k}(0))
	\end{equation*}
	almost surely, where $(\bar{X}_{k}(0),\bar{Y}_{k}(0),\bar{Z}_{k}(0))\in\{
		(x,y,z)\in[0,1]^{3}:x+y+z=p_k\}$.
\end{assumption}

We denote by $D^{k}$ the space of right-continuous with left limit functions $f:\R_+\to\R^k$. 

\begin{theorem}[Functional Law of Large Numbers]
	\label{FLLN} Assume that assumptions \ref{def:TGC} and \ref{initialconv} hold, then, as $n$ goes to infinity,
	\begin{equation*}
		\frac{1}{\#V^{n}}(X_{1}^{n},Y_{1}^{n},Z_{1}^{n},\dots,\bar{X}_{\ns}^{n}
		,\bar{Y}_{\ns}^{n},\bar{Z}_{\ns}^{n})\to(X_{1},Y_{1},Z_{1},\dots,\bar
		{X}_{\ns},\bar{Y}_{\ns},\bar{Z}_{\ns})\in D^{3\ns}
	\end{equation*}
	in probability, where
	\begin{eqnarray*}
		\bar{X}_k(t)&=&p_k-\bar{Y}_k(t)-\bar{Z}_k(t),\\
		\bar{Y}_k(t)&=&\bar{Y}_k(0)F^{c}(t)+\int_0^t\lambda\,F^c(t-s)\bar{X}_k(s)\sum_{j=1}^{\ns}\frac{n_k(j)}{p_j}\bar{Y}_j(s)ds\\
		&&-\int_0^t\lambda\bar{Y}_k(s)\sum_{j=1}^{\ns}\frac{n_{k}(j)}{p_{j}}(p_{j}-\bar{X}_j(s))ds, \mbox{ and }\\
		\bar{Z}_k(t)&=&\bar{Y}_k(0)F(t)+\bar{Z}_k(0)+\int_0^t\lambda\,F(t-s)\bar{X}_k(s)\sum_{j=1}^{\ns}\frac{n_{k}(j)}{p_{j}}\bar{Y}_j(s)ds\\
		&&+\int_0^t\lambda\bar{Y}_k(s)\sum_{j=1}^{\ns}\frac{n_{k}(j)}{p_{j}}(p_j-\bar{X}_j(s))ds,
	\end{eqnarray*}
    for any $k\in\{1,\dots,\ns\}$, where $F^c=1-F$ and $p_k:=\lim_{n\to\infty}\#V^{n}_{k}/\#V^{n}$.
\end{theorem}

Now we look into the fluctuations. Let
\begin{align*}
	\left\{\begin{array}{c}\hat{X}_k^n(t):=\sqrt{\#V^{n}}\left(\frac{1}{\#V^{n}}X_k^n(t)-\bar{X}_k(t)\right),\\ \hat{Y}_k^n(t):=\sqrt{\#V^{n}}\left(\frac{1}{\#V^{n}}Y_k^n(t)-\bar{Y}_k(t)\right), \\ \hat{Z}_k^n(t):=\sqrt{\#V^{n}}\left(\frac{1}{\#V^{n}}Z_k^n(t)-\bar{Z}_k(t)\right).\end{array}\right. & \text{ for }k=1,\dots,\ns, &
\end{align*}

As before, we need the convergence of the initial conditions.

\begin{assumption}
	\label{initialdist} There are random variables $\hat{X}(0)$, $\hat{Y}(0)$
	and $\hat{Z}(0)$ such that
	\begin{equation*}
		(\hat{X}_{1}^{n}(0),\hat{Y}_{1}^{n}(0),\hat{Z}_{1}^{n}(0),\dots,\hat{X}
		_{\ns}^{n}(0),\hat{Y}_{\ns}^{n}(0),\hat{Z}_{\ns}^{n}(0))\Rightarrow(\hat
			{X}_{1}(0),\hat{Y}_{1}(0),\hat{Z}_{1}(0),\dots,\hat{X}_{\ns}(0),\hat{Y}
			_{\ns}(0),\hat{Z}_{\ns}(0)),
	\end{equation*}
	as $n$ goes to infinity, and
	\begin{equation*}
		\sum_{k=1}^{\ns}\sup_{n}\E[\hat{X}_{k}^{n}(0)^{2}]+\sup_{n}\E[\hat{Y}
			_{k}^{n}(0)^{2}]+\sup_{n}\E[\hat{Z}_{k}^{n}(0)^{2}]<\infty.
	\end{equation*}
\end{assumption}

\begin{theorem}[Functional Central Limit Theorem]
	\label{FCLT} Under the hypothesis of Theorem \ref{FLLN} and Assumption \ref{initialdist}, as $n$ goes to infinity,
	\begin{equation*}
		(\hat{X}_{1}^{n},\hat{Y}_{1}^{n},\hat{Z}_{1}^{n},\dots,\hat{X}_{\ns}^{n}
		,\hat{Y}_{\ns}^{n},\hat{Z}_{\ns}^{n})\Rightarrow(\hat{X}_{1},\hat{Y}_{1}
		,\hat{Z}_{1},\dots,\hat{X}_{\ns},\hat{Y}_{\ns},\hat{Z}_{\ns})\in D^{3\ns}
		,
	\end{equation*}
	in probability, where
	\begin{eqnarray*}
		\hat{X}_k(t)&=&-\hat{Y}_k(t)-\hat{Z}_k(t),\\
		\hat{Y}_k(t)&=&\hat{Y}_k(0)F^c(t)+\hat{Y}_{k}^{0}(t)+\sum_{j=1}^{\ns}\hat{Y}_{k,j}(t)+\sum_{j=1}^{\ns}\hat{B}_{k,j}(t)\\
		&&+\sum_{j=1}^{\ns}\frac{n_{k}(j)}{p_{j}}\int_0^t\lambda F^c(t-s)\big(\hat{X}_k(s)\bar{Y}_j(s)+\bar{X}_k(s)\hat{Y}_j(s)\big)ds\\
		&&-\sum_{j=1}^{\ns}\frac{n_{k}(j)}{p_{j}}\int_0^t\lambda\Big(\hat{Y}_k(s)\big(p_j-\bar{X}_j(s)\big)-\bar{Y}_k(s)\hat{X}_j(s)\Big)ds,\\
		\hat{Z}_k(t)&=&\hat{Z}_k(0)+\hat{Z}_k^{0}(t)+\sum_{j=1}^{\ns}\hat{Z}_{k,j}(t)-\sum_{j=1}^{\ns}\hat{B}_{k,j}(t)\\
		&&+\sum_{j=1}^{\ns}\frac{n_{k}(j)}{p_{j}}\int_0^t\lambda F(t-s)\big(\hat{X}_k(s)\bar{Y}_j(s)+\bar{X}_k(s)\hat{Y}_j(s)\big)ds\\
		&&+\sum_{j=1}^{\ns}\frac{n_{k}(j)}{p_{j}}\int_0^t\lambda\Big(\hat{Y}_k(s)\big(p_j-\bar{X}_j(s)\big)-\bar{Y}_k(s)\hat{X}_j(s)\Big)ds,
	\end{eqnarray*}
	for any $k\geq1$, where
	$\hat{Y}_{k}^{0}(t),\hat{Y}_{k,j}(t),\hat{B}_{k,j}(t),\hat{Z}_{k}^{0}(t)$
	and $\hat{Z}_{k,j}(t)$ are zero-mean Gaussian processes for any
	$k,j\in\{1,\dots,\ns\}$.
\end{theorem}

\begin{lemma} \label{lemma:cov}
	The nonnull covariances between the zero-mean Gaussian processes are given
	in Table \ref{TableOfCov}
\end{lemma}

\begin{table}[h]
	\centering
\begingroup
\setlength{\tabcolsep}{10pt}     
\renewcommand{\arraystretch}{1.35} 
\setlength{\extrarowheight}{2pt}   
\newcommand{\Bstrut}{\rule[-3ex]{0pt}{7.2ex}}
	\begin{tabular}{|l|l|}
		\hline
		\textbf{Covariance}                                            & \textbf{Result}                                                                                                                                            \\
		\hline
		$\mathrm{Cov}[\widehat{Y}_{k}^{0}(t), \widehat{Y}_{k}^{0}(r)]$ & $\bar{Y}_{k}(0)\,F(t \wedge r)\big(1-F(t \vee r)\big)$                                                                                                     \\
		\hline
		$\mathrm{Cov}[\widehat{Z}_{k}^{0}(t), \widehat{Z}_{k}^{0}(r)]$ & $\bar{Y}_{k}(0)\,F(t \wedge r)\big(1-F(t \vee r)\big)$                                                                                                     \\
		\hline
		$\mathrm{Cov}[\widehat{Y}_{k}^{0}(t), \widehat{Z}_{k}^{0}(r)]$ & $-\bar{Y}_{k}(0)\,F(t \wedge r)\,F^{c}(t \vee r)$                                                                                                          \\
		\hline
		$\mathrm{Cov}[\widehat{Y}_{k,j}(t), \widehat{Y}_{k,j}(r)]$     & $\displaystyle \int_{0}^{t\wedge r}\!\lambda\,c_{k,j}\,F^{c}((t\wedge r)-s)\,\bar{X}_{k}(s)\,\bar{Y}_{j}(s)\,ds\Bstrut$                                          \\
		
		\hline
		$\mathrm{Cov}[\widehat{Z}_{k,j}(t), \widehat{Z}_{k,j}(r)]$     & $\displaystyle \int_{0}^{t\wedge r}\!\lambda\,c_{k,j}\,F((t\wedge r)-s)\,\bar{X}_{k}(s)\,\bar{Y}_{j}(s)\,ds\Bstrut$                                            \\
		\hline
		$\Cov\!\big[\widehat Y_{k,j}(t),\,\widehat Z_{k,j}(r)\big]$    & $-\displaystyle \int_{0}^{t\wedge r}\!\lambda\,c_{k,j}\,F^{c}((t\wedge r)-s)\,\bar{X}_{k}(s)\,\bar{Y}_{j}(s)\,ds\Bstrut$ \\
		\hline
	\end{tabular}
\endgroup
	\caption{Covariances between the processes obtained in Theorem
		\ref{FCLT}, for $k\in\{1,\dots,\ns\}.$}
	\label{TableOfCov}
\end{table}

\section{Proofs}
We divide this section into four parts. We first derive the stochastic intensities for the sequence of counting processes defined in graphs satisfying the QTS property. Secondly, we prove the Functional Law of Large Numbers. In the third part, we prove the Functional Central Limit Theorem. Finally, we show the results for infinite quasi-transitive graphs of subexponential growing.

\subsection{Stochastic intensity for sequences satisfying the QTS property}

From now until further mention, $(\Gamma^n)_{n\geq1}$ will be a sequence satisfying the QTS property. \\

Recall that $A^{n,k}$ is the counting process that takes into account the ignorant--spreader interactions of type $k$, that is, the process is the sum of the processes $A^{v}$ that take into account the ignorant--spreader interactions of the vertex $v\in V_{k}^{n}$.

Notice that for $v\in V_{k}^{n}$, $A^{v}$ writes as
\begin{equation*}
	A^{v}(t)=\sum_{y:y\sim v}\sum_{i=1}^{N_{v,y}(t)}\one_{\{v\in\calX^{n}(\theta_{v,y}^{i}-)\}}
	\one_{\{y\in\calY^{n}(\theta_{v,y}^{i}-)\}}.
\end{equation*}

We may also consider the edge interactions for a neighbor $u$ of $v$:
\begin{equation*}
	A^{v,u}(t)=\sum_{i=1}^{N_{v,u}(t)}\one_{\{v\in\calX^{n}(\theta_{v,u}^{i}-)\}}
	\one_{\{u\in\calY^{n}(\theta_{v,u}^{i}-)\}}.
\end{equation*}

We consider a dynamics which does not distinguish between vertices of the same type. To this end, we consider a filtration, $\F^{n}=(\F^{n}_{t})_{t\geq0}$, that takes into account the available information at any instant $t\geq0$: the type of any individual and the quantities of individuals by type,
\begin{equation*}
	\F_{t}^{n}:=\sigma\langle V^{n},E^{n},\#V_{k}^{n},\{X^{n}_{k}(s),Y^{n}_{k}
	(s),Z^{n}_{k}(s):0\leq s\leq t\}: \ 1\leq k\leq \ns\rangle.
\end{equation*}
Observe that the $\sigma$-algebra does not have any information about the
local dynamics, that is, the state of any particular vertex is unknown.

Notice that the initial configuration and the Poisson process clocks fully describe the process, then we characterize the sample space by:
\begin{equation*}
	\Omega^{(n)}=S^{V^n}\times(\R^{E^n})^{\bbN},
\end{equation*}
where every $\R$-coordinate represents a sojourn time of a Poisson clock related to an edge.

Let $G^{n}=\Aut(\Gamma^{n})$ and let $g\in G^{n}$. Observe that we may induce
a map $g^{\#}:\Omega\to\Omega$ by
\begin{equation*}
	(\eta,(\iota_{u,v}^{i})_{i\geq1,\{u,v\}\in E^{n}})\mapsto(g^{\#}_{1}(\eta
	),(g^{\#}_{2}(\iota_{u,v}^{i}))_{i\geq1,\{u,v\}\in E^{n}}),
\end{equation*}
where $g_{1}^{\#}(\eta)(v)=\eta(g^{-1}v)$ for $v\in V^{n}$ and $g_{2}^{\#}(\iota_{u,v}^{i})=\iota^{i}_{g^{-1}u,g^{-1}v}$. As $g$ is well defined, so is $g^{\#}$. By definition, we have the following property.

\begin{proposition}[Invariance of measure under automorphisms]\label{invariance}
  Suppose that the law of $\eta_0$ is $G^n$–invariant, and that the family of Poisson clocks has a product law which is invariant under the induced action $g^\#$, i.e., $\prob(\eta_{0}\times B)=\prob
		(g^{-1}\eta_{0}\times B)$. Then, for every $g\in G^n$,
  \[
    \mathbb{P}\circ (g^\#)^{-1} = \mathbb{P}.
  \]
\end{proposition}

As a consequence of Proposition \ref{invariance}, for any $f:S^{V^n}\to\R$,
\begin{equation*}
	\E[f(\eta_{t})|\F_{t}]=\E^{\#}[(f\circ g^{\ast})(\eta_{t})|\F_{t}]=\E[f(g
			^{\#}_{1}\eta_{t})|\F_{t}],
\end{equation*}
where $E^{\#}$ denotes the expectation with respect to $\prob^{\#}$.

Then, it holds
\begin{equation*}
	\E[\one_{\{\eta_t(v)=X\}}\one_{\{\eta_t(w)=Y\}}|\F_{t}^{n}]=\E[\one_{\{\eta_t(gv)=X\}}
		\one_{\{\eta_t(gw)=Y\}}|\F_{t}^{n}]=:a_{k,j}^{n}(t).
\end{equation*}
That is,
\begin{eqnarray*}
	a_{k,j}^n(t)&=&\frac{\sum_{v\in V_k^n}\sum_{w\in V_j^n}\E[\one_{\{\eta_t(v)=X\}}\one_{\{\eta_t(w)=Y\}}|\F_{t}^{n}]}{\#V_{k}^{n}\#V_{j}^{n}}\\
	&=&\frac{\E[\sum_{v\in V_k^n}\sum_{w\in V_j^n}\one_{\{\eta_t(v)=X\}}\one_{\{\eta_t(w)=Y\}}|\F_{t}^{n}]}{\#V_{k}^{n}\#V_{j}^{n}}\\
	&=&\frac{E[X_{k}^{n}Y_{j}^{n}|\F_{t}^{n}]}{\#V_{k}^{n}\#V_{j}^{n}}\\
	&=&\frac{X_{k}^{n}Y_{j}^{n}}{\#V_{k}^{n}\#V_{j}^{n}}.
\end{eqnarray*}

For any $C(t)$ nonnegative, $\F^{n}_{t}$–predictable process, we have 
\[
	\E\Bigl[\!\int_{0}^{\infty}C(t)\,A^{v,u}(dt)\Bigr] = \E\Bigl[\!\int_{0}^{\infty}
		C(t)\, \lambda\,\one_{\{v\in\tilde X(t-)\}}\,\one_{\{u\in\tilde Y(t-)\}}\,
		dt\Bigr].
\]
Since $C(t)\,\lambda$ is $\F^{n}_{t-}$–measurable, the tower property of conditional
expectation gives, for each $t$,
\[
	\E\bigl[C(t)\,\lambda\,\one_{\{v\in\tilde X(t-)\}}\one_{\{u\in\tilde Y(t-)\}}
		\bigr] = \E\Bigl[C(t)\,\lambda\, \E\bigl[\one_{\{v\in\tilde X(t-)\}}\one_{\{u\in\tilde
				Y(t-)\}}\mid \F^{n}_{t-}\bigr]\Bigr],
\]
and
\[
	\E\bigl[\one_{\{v\in\tilde X(t-)\}}\one_{\{u\in\tilde Y(t-)\}}\mid \F^{n}
		_{t-}\bigr] = \frac{X^{n}_{k}(t)}{\#V^{n}_{k}}\;\frac{Y^{n}_{j}(t)}{\#V^{n}_{j}}
	.
\]
Hence, by Fubini’s theorem,
\[
	\E\Bigl[\int_{0}^{\infty}C(t)\,A^{v,u}(dt)\Bigr] = \E\Bigl[\!\int_{0}^{\infty}
	C(t)\, \lambda\, \frac{X^{n}_{k}(t)}{\#V^{n}_{k}}\, \frac{Y^{n}_{j}(t)}{\#V^{n}_{j}}
	\,dt\Bigr].
\]

We summarize the computations above in the next result.

\begin{lemma}
	\label{stocintv} The $\F_{t}^{n}$-stochastic intensity of $A^{v,u}$ for $v\in V_{k}^{n}$, $u\in V_{j}^{n}$ is given by
	\begin{equation}
		\one_{\{v\sim u\}}\lambda\frac{X_{k}^{n}(t)}{\#V_{k}^{n}}\frac{Y_{j}^{n}(t)}{\#V_{j}^{n}}
		.
	\end{equation}
\end{lemma}

It follows that the stochastic intensity of process $A^{v}(t)$, where
$v\in V_{k}^{n}$, is 
\begin{equation*}
	\lambda\frac{X_{k}^{n}(t)}{\#V_{k}^{n}}\sum_{j=1}^{\ns}n_{k}(j)\frac{Y_{j}^{n}(t)}{\#V_{j}^{n}}
	,
\end{equation*}
where $n_{k}(j)$ is the number of neighbors of type $j$ from the vertex $v$.
Now, since the local topologies of vertices of the same type are indistinguishable,
we sum over $v\in V_{k}$ and obtain the $\F^{n}$-stochastic intensity of
$A^{n,k}(t)$,
\begin{equation*}
	\lambda X_{k}^{n}(t)\sum_{j=1}^{\ns}n_{k}(j)\frac{Y_{j}^{n}(t)}{\#V_{j}^{n}}.
\end{equation*}

In a similar fashion, we obtain the $\F^{n}$-stochastic intensity of $B^{n,k}
	(t)$,
\begin{equation*}
	\lambda Y_{k}^{n}(t)\sum_{j=1}^{\ns}n_{k}(j)\left(1-\frac{X_{j}^{n}(t)}{\#V_{j}^{n}}
	\right).
\end{equation*}
\subsection{Functional Law of Large Numbers}
\label{FLLN1stcase}
\begin{lemma}
	\label{tightness} For any $k\in\{1,\dots,\ns\}$, the sequences of
	processes $\{A^{n,k}/\#V^{n}\}$ and $\{B^{n,k}/\#V^{n}\}$ are tight.
\end{lemma}
\proof Without loss of generality, observe that for the stochastic intensities of
$A^{n,k}$, $\gamma_{n}$, $\sup_{n}\gamma_{n}/\#V^{n}$ are bounded almost
surely. Then, we apply an analogue version of Lemma 3.1 in \cite{CL25} and
the result follows. \endproof

Now define, for any $k\in\{1,\dots,\ns\}$,
\begin{equation*}
	\bar{Y}_{k,0}^{n}(t):=\frac{1}{\#V^{n}}\sum_{i=1}^{Y_k^n(0)}\one_{\{t<\eta_{i}^{k,0}\}}
	,\quad t\geq0.
\end{equation*}

Recall that $F^c=1-F$.

\begin{lemma}
	\label{convinit} As $n$ goes to infinity,
	\begin{equation*}
		\bar{Y}_{k,0}^{n}\to\bar{Y}_{k}(0)F^{c}\in D.
	\end{equation*}
\end{lemma}
\begin{proof}
	Let
	\begin{equation*}
		\check{Y}_{k,0}^{n}:=\frac{1}{\#V^{n}}\sum_{i=1}^{\lfloor\#V^n\times\bar{Y}_k(0)\rfloor}
		\one_{\{t<\eta_{i}^{k,0}\}},\quad t\geq0.
	\end{equation*}
	Then
	\begin{equation*}
		|\bar{Y}_{k,0}^{n}(t)-\check{Y}_{k,0}^{n}(t)|\leq\frac{1}{\#V^{n}}\sum
		_{i=Y_k^n(0)\wedge \lfloor \#V^n\times\bar{Y}_k(0)\rfloor}^{Y_k^n(0)\vee
			\lfloor \#V^n\times\bar{Y}_k(0)\rfloor}\one_{\{t<\eta_{i}^{k,0}\}},\quad
		t\geq0.
	\end{equation*}
	Since the random variables $\eta_{i}^{k,0}$ are i.i.d. and independent
	of $\F_{0}^{n}$,
	\begin{equation*}
		\E\Bigg[\frac{1}{\#V^{n}}\sum_{i=Y_k^n(0)\wedge \lfloor \#V^n\times\bar{Y}_k(0)\rfloor}
		^{Y_k^n(0)\vee \lfloor \#V^n\times\bar{Y}_k(0)\rfloor}\one_{\{t<\eta_{i}^{k,0}\}}
		\Bigg|\F_{0}^{n}\Bigg]\leq F^{c}(t)|\bar{Y}^{n}_{k}(0)-\bar{Y}_k(0)|
		,
	\end{equation*}
	which by Assumption \ref{initialconv} converges in probability to zero as
	$n$ goes to infinity. As a corollary of Donsker's Theorem for
	empirical processes (Theorem 14.3 in \cite{Billingsley1999}) we have,
	\begin{equation*}
		\bar{Y}_{k,0}^{n}\to\bar{Y}_{k}(0)F^{c}\in D,
	\end{equation*}
	as $n$ goes to infinity.
\end{proof}

For any $k\in\{1,\dots,\ns\}$, we write
\begin{equation*}
	\bar{Y}_{k}^{n}(t):=\frac{1}{\#V^{n}}\sum_{j=1}^{A^{n,k}(t)}\one_{\{t<\tau_j^{n,k}+\eta_j^k\}}
	,\quad t\geq0.
\end{equation*}

Applying an analogue version of Theorem 3.4 in \cite{CL25}, we obtain, for
any $T>0$,
\begin{equation*}
	\sup_{t\in[0,T]}|\bar{Y}_{k}^{n}(t)-\E[\bar{Y}_{k}^{n}(t)|A^{n,k}(t)]|\to
	0,
\end{equation*}
in probability as $n$ goes to infinity, for any $k$.

As $\{A^{n,k}/\#V^{n}\}$ and $\{B^{n,k}/\#V^{n}\}$ are tight for any $k$, we work with a convergent subsequence of 
\begin{equation*}
    \frac{1}{\#V^{n}}(A^{n,1},B^{n,1},\dots,A^{n,\ns},B^{n,\ns})
\end{equation*}
and let $(A^{1},B^{1},\dots,A^{\ns},B^{\ns})$ be the limit along this convergent subsequence.

\begin{lemma}
	We have, as $n$ goes to infinity,
	\begin{equation*}
		\E[\bar{Y}_{k}^{n}(t)|A^{n,k}(t)]\to\bar{Y}_{k}(t):=\int_{0}^{t}F^{c}
		(t-s)d\bar{A}^{k}(s),
	\end{equation*}
	in probability, for any $k$.
\end{lemma}
\proof We have
\begin{equation*}
	\E[\bar{Y}_{k}^{n}(t)|A^{n,k}(t)]=\frac{1}{\#V^{n}}\sum_{j=1}^{A^{n,k}(t)}
	F^{c}(t-\tau_{j}^{n,k})=\int_{0}^{t}F^{c}(t-s)d\bar{A}^{n,k}(s),\quad t\geq
	0.
\end{equation*}
Integration by parts holds
\begin{equation*}
	\E[\bar{Y}_{k}^{n}(t)|A^{n,k}(t)]=\bar{A}^{n,k}(t)-\int_{0}^{t}\bar{A}^{n,k}
	(s)dF^{c}(t-s).
\end{equation*}
As a consequence of Theorem 13.4 and discussion in Section 16 in \cite{Billingsley1999},
$\bar{A}^{k}$ is almost surely continuous. Letting $n\to\infty$, we obtain
\begin{equation*}
	\E[\bar{Y}_{k}^{n}(t)|A^{n,k}(t)]\to\bar{A}^{k}(t)-\int_{0}^{t}\bar{A}^{k}
	(s)dF^{c}(t-s)=\tilde{Y}_{k}^{n}(t),\quad t\geq0.
\end{equation*}
\endproof

\begin{proposition}
	For each $i\in\{1,\dots,\ns\}$, the proportion $p_{i}^{n}$ is constant
	in $n$.
\end{proposition}

\begin{proof}
	Fix $i,j$. Let
	\[
		E(V_{i}^{n},V_{j}^{n}):=\bigl\{\{u,w\}\in E^{n}:\ u\in V_{i}^{n},\ w\in
		V_{j}^{n}\bigr\}.
	\]
	Counting incidences from $V_{i}^{n}$ gives
	$|E(V_{i}^{n},V_{j}^{n})|=|V_{i}^{n}|\cdot n_{i}(j)$, while counting from
	$V_{j}^{n}$ gives $|E(V_{i}^{n},V_{j}^{n})|=|V_{j}^{n}|\cdot n_{j}(i)$.
	Hence
	\[
		\frac{|V_{i}^{n}|}{|V_{j}^{n}|}=\frac{n_{j}(i)}{n_{i}(j)}=:C_{ij},
	\]
	which is independent of $n$ by assumption. Therefore
	\[
		p_{i}^{n}=\frac{|V_{i}^{n}|}{|V^{n}|}=\frac{C_{i1}\,|V_{1}^{n}|}{\sum_{t=1}^{\ns}C_{t1}\,|V_{1}^{n}|}
		=\frac{C_{i1}}{\sum_{t=1}^{\ns}C_{t1}},
	\]
	showing that $p_{i}^{n}$ does not depend on $n$.
\end{proof}

\begin{lemma}
	\label{ABlimit} As $n$ goes to infinity, it holds
	\begin{align*}
		\frac{1}{\#V^{n}}\Bigg(A^{n,k}(t)-\int_{0}^{t}\lambda X_{k}^{n}(s)\sum_{j=1}^{\ns}n_{k}(j)\frac{Y_{j}^{n}(s)}{\#V^{n}_{j}}ds\Bigg)\to0, \\
		\frac{1}{\#V^{n}}\bigg(B^{n,k}(t)-\int_{0}^{t}\lambda Y_{k}^{n}(s)\sum_{j=1}^{\ns}n_{k}(j)\bigg(1-\frac{X_{j}^{n}(s)}{\#V^{n}_{j}}\bigg)ds\Bigg)\to0
	\end{align*}
	in probability, for any $k\in\{1,\dots,\ns\}$.
\end{lemma}
\proof Without loss of generality, we fix $k$ and show the first case. Observe
that the integrand is the $\F^{n}$-stochastic intensity of $A^{n,k}$ so
\begin{equation*}
	\hat{A}^{n,k}(t):=\frac{1}{\#V^{n}}\Bigg(A^{n,k}(t)-\int_{0}^{t}\lambda X
		_{k}^{n}(s)\sum_{j=1}^{\ns}n_{k}(j)\frac{Y_{j}^{n}(s)}{\#V^{n}_{j}}ds\Bigg
	)
\end{equation*}
is an $\F^{n}$-martingale (see Chapter II in \cite{Bremaud1981}). Since the quadratic variation of a stochastic integral with respect to the Lebesgue measure is zero, we have:
\begin{equation}
	\label{QVA}[\hat{A}^{n,k}](t)=\bigg[\frac{A^{n,k}}{\#V^{n}}\bigg](t)=\frac{A^{n,k}(t)}{(\#V^{n})^{2}},
\end{equation}
where the last equality is a consequence of $A^{n,k}$ being a counting process. Taking the limit of $n$ going to infinity, Equation \eqref{QVA} goes to zero and the result follows by applying Theorem 1.4 (Chapter 7) in \cite{Kurtz}.
\endproof

Notice that along the convergent subsequence, the limit
$(1/\#V^{n})(X^{n}_{1},Y^{n}_{1},Z^{n}_{1},\dots,X^{n}_{\ns},Y^{n}_{\ns},Z^{n}
	_{\ns})$
is given by
\begin{eqnarray*}
	\tilde{X}_k(t)&=&p_k-\tilde{Y}_k(t)-\tilde{Z}_k(t),\\
	\tilde{Y}_k(t)&=&\bar{Y}_{k,0}(0)F^c(t)+\bar{Y}_k(t)-\bar{B}^{k}(t),\\
	\tilde{Z}_k(t)&=&\bar{Y}_{k,0}(0)F(t)+\bar{Z}_{k}(0)+\int_0^tF(t-s)d\bar{A}^k(s)+\bar{B}^k(t),
\end{eqnarray*}
for $t\geq0$ and $k\in\{1,\dots,\ns\}$. By Lemma \ref{ABlimit}, we may
characterize the limits and write
\begin{eqnarray*}
	\tilde{X}_k(t)&=&p_k-\tilde{Y}_k(t)-\tilde{Z}_k(t),\\
	\tilde{Y}_k(t)&=&\bar{Y}_{k,0}(0)F^c(t)+\int^tF^c(t-s)\lambda \tilde{X}_k(s)\sum_{j=1}^{\ns}\frac{n_{k}(j)}{p_{j}}\tilde{Y}_j(s)ds\\
	&&-\int^t\lambda \tilde{Y}_k(s)\sum_{j=1}^{\ns}n_k(j)\bigg(1-\frac{\tilde{X}_{j}(s)}{p_{j}}\bigg)ds,\\
	\tilde{Z}_k(t)&=&\bar{Y}_{k,0}(0)F(t)+\bar{Z}_{k}(0)+\int_0^tF(t-s)\lambda
	\tilde{X}_k(s)\sum_{j=1}^{\ns}\frac{n_{k}(j)}{p_{j}}\tilde{Y}_j(s)ds\\
	&&+\int^t\lambda \tilde{Y}_k(s)\sum_{j=1}^{\ns}n_k(j)\bigg(1-\frac{\tilde{X}_{j}(s)}{p_{j}}\bigg)ds.
\end{eqnarray*}

As the limit is the same for any convergent subsequence, the proof of \ref{FLLN} is complete.

\subsection{Functional Central Limit Theorem}
For a fixed $k$, we write
\begin{eqnarray*}
	\hat{X}^n_k(t)&=&-\hat{Y}^n_k(t)-\hat{Z}_k^n(t),\\
	\hat{Y}^n_k(t)&=&\hat{Y}_k^n(0)F^c(t)+\hat{Y}_{k,0}^n(t)+\sum_{j=1}^{\ns}\hat{Y}^n_{k,j}(t)+\sum_{j=1}^{\ns}\hat{B}^n_{k,j}(t)\\
	&&+\sum_{j=1}^{\ns}\frac{n_{k}(j)}{p_{j}}\int_0^t\lambda F^c(t-s)\bigg(\hat{X}_k^n(s)\frac{Y^{n}_{j}(s)}{\#V^{n}}+\bar{X}_k(s)\hat{Y}^n_j(s)\bigg)ds\\
	&&-\sum_{j=1}^{\ns}\frac{n_{k}(j)}{p_{j}}\int_0^t\lambda\bigg(\hat{Y}_k^n(s)\Big(p_j-\frac{X^{n}_{j}(s)}{\#V^{n}}\Big)-\bar{Y}_k(s)\hat{X}^n_j(s)\bigg)ds,\\
	\hat{Z}^n_k(t)&=&\hat{Z}_k^n(0)+\hat{Z}_{k,0}^n(t)+\sum_{j=1}^{\ns}\hat{Z}^n_{k,j}(t)-\sum_{j=1}^{\ns}\hat{B}^n_{k,j}(t)\\
	&&+\sum_{j=1}^{\ns}\frac{n_{k}(j)}{p_{j}}\int_0^t\lambda F(t-s)\bigg(\hat{X}_k^n(s)\frac{Y^{n}_{j}(s)}{\#V^{n}}+\bar{X}_k(s)\hat{Y}^n_j(s)\bigg)ds\\
	&&+\sum_{j=1}^{\ns}\frac{n_{k}(j)}{p_{j}}\int_0^t\lambda\bigg(\hat{Y}_k^n(s)\Big(p_j-\frac{X^{n}_{j}(s)}{\#V^{n}}\Big)-\bar{Y}_k(s)\hat{X}^n_j(s)\bigg)ds,
\end{eqnarray*}
where
\begin{eqnarray}
	\hat{Y}_{k,0}^n(t)&:=&\frac{1}{\sqrt{\#V^{n}}}\sum_{i=0}^{Y_k(0)}(\one_{\{t<\eta_i^{k,0}\}}-F^c(t)),\\
	\hat{Z}_{k,0}^n(t)&:=&\frac{1}{\sqrt{\#V^{n}}}\sum_{i=0}^{Y_k(0)}(\one_{\{t\geq\eta_i^{k,0}\}}-F^c(t)),\\
	\hat{Y}_{k,j}^{n}(t)&:=&\sqrt{\#V^{n}}\bigg(\frac{1}{\#V^{n}}\sum_{i=1}^{A^{n}_{k,j}(t)}\one_{\{t<\tau_{i}^{n,k,j}+\eta_i^k\}}-\int_0^t\lambda\frac{n_{k}(j)}{p_{j}}F^c(t-s)\frac{X_k^n}{\#V^{n}}(s)\frac{Y^n_j}{\#V^{n}}(s)\bigg),\\
	\hat{Z}_{k,j}^{n}(t)&:=&\sqrt{\#V^{n}}\bigg(\frac{1}{\#V^{n}}\sum_{i=1}^{A^{n}_{k,j}(t)}\one_{\{t\geq\tau_{i}^{n,k,j}+\eta_i^k\}}-\int_0^t\lambda\frac{n_{k}(j)}{p_{j}}F(t-s)\frac{X_k^n}{\#V^{n}}(s)\frac{Y_j^n}{\#V^{n}}(s)\bigg),\\
	\hat{B}_{k,j}^{n}(t)&:=&\sqrt{\#V^{n}}\bigg(\frac{1}{\#V^{n}}B^{n}_{k,j}(t)-\int_0^t\lambda\frac{n_{k}(j)}{p_{j}}\frac{Y_k}{\#V^{n}}(s)\big(p_j-\frac{X_j}{\#V^{n}}(s)\big)\bigg).
\end{eqnarray}
Here $A^{n}_{k,j}$ (respectively, $B^{n}_{k,j}$) is the restriction of the
counting process $A^{n}_{k}$ (respectively, $B^{n}_{k}$) to the contacts between
$V^{n}_{k}$ and $V^{n}_{j}$.

\begin{lemma}
	\label{Convzerohat} As $n$ goes to infinity,
	\begin{equation*}
		(\hat{Y}^{n}_{1,0},\hat{Z}^{n}_{1,0},\dots,\hat{Y}^{n}_{\ns},\hat{Z}^{n}
		_{\ns})\Rightarrow (\hat{Y}_{1,0},\hat{Z}_{1,0},\dots,\hat{Y}_{\ns},\hat
		{Z}_{\ns}),
	\end{equation*}
	where the covariances are given in Table \ref{TableOfCov}.
\end{lemma}
\begin{proof}
	It follows by applying Donsker's Theorem for empirical processes (Theorem $14.3$ in \cite{Billingsley1999}).
    In order to compute the covariances, observe that the independence of the $\eta_{i}$'s gives:

\begin{align*}
	\mathrm{Cov}[\widehat{Z}_{k}^{0}(t), \widehat{Z}_{k}^{0}(r)] & = \lim_{n\to\infty}\frac{\lfloor \#V^n\bar{Y}_{k}(0)\rfloor}{\#V^n} \cdot \mathrm{Cov}[\mathbf{1}\{\eta \le t\}, \mathbf{1}\{\eta \le r\}]                 \\
	& = \bar{Y}_{k}(0) \big( \mathbb{P}(\eta \le t \text{ and }\eta \le r) - \mathbb{P}(\eta \le t)\mathbb{P}(\eta \le r) \big) \\
	& = \bar{Y}_{k}(0) \big( F(t \wedge r) - F(t)F(r) \big)\\
    & = \bar{Y}_{k}(0)\, F(t \wedge r)\, F^c(t \vee r).
\end{align*}
Observe that 
\begin{equation*}
    \one_{\{\eta_i>t\}}-F^c(t)=\big((1-\one_{\{\eta_i\leq t\}}\big)-\big(1-F(t)\big)=-\big(\one_{\{\eta_i\leq t\}}-F(t)\big),
\end{equation*}
so $\hat{Y}_k^0(t)=-\hat{Z}_k^0(t)$ and we repeat the computation above to obtain the result for both cases.
\end{proof}

The next result is the analogue of Proposition $3.10$ in \cite{CL25} in our
case.

\begin{proposition}
	\label{Propstocint} The process
	\begin{equation*}
		\sum_{i=1}^{A^{n}_{k,j}(t)}\one_{\{t<\tau_{i}^{n,k,j}+\eta_i^k\}}
	\end{equation*}
	has $\F_{t}^{n}$-stochastic intensity
	\begin{equation}
		\label{stinpr}\int_{0}^{t}\lambda\frac{n_{k}(j)}{p_{j}}F^{c}(t-s)\bar
		{X}_{k}(s)\bar{Y}_{j}(s)ds.
	\end{equation}
\end{proposition}
\begin{proof}
	The tower property of the conditional expectation given the counting process
	$A^{n}_{k,j}(t)$ and the measurability of \eqref{stinpr} give us the
	result.
\end{proof}

As Proposition \ref{Propstocint} may be applied to $\hat{Y}_{k,j}^{n}$, it is also applicable to $\hat{Z}_{k,j}^{n}$. In such a way, we derive the
following lemma.

\begin{lemma}\label{LocalMartingale}
    For any $k,j$, the processes $\hat{Y}_{k,j}^{n},\hat{Z}_{k,j}^{n}$ and $\hat{B}_{k,j}^{n}$ are zero-mean $\F_{t}^{n}$-local martingales.
\end{lemma}
\begin{proof}
	From Proposition \ref{Propstocint} and the subsequent paragraph, we apply the item $(\alpha)$ of Theorem $T8$ in \cite{Bremaud1981} and obtain the result. 
\end{proof}

Now we characterize the quadratic variation of the processes.
\begin{lemma}
	\label{Convquadvar} For any $k,j$, there exists $\mathcal{W}_{k,j}=([\hat
			{Y}_{k,j}],[\hat{Z}_{k,j}],[\hat{B}_{k,j}],[\hat{Y}_{k,j},\hat{Z}_{k,j}],[\hat{Y}_{k,j},\hat{B}_{k,j}],[\hat{Z}_{k,j},\hat{B}_{k,j}]))$, such that
	\begin{equation*}
		([\hat{Y}_{k,j}^{n}],[\hat{Z}_{k,j}^{n}],[\hat{B}_{k,j}^{n}],[\hat{Y}_{k,j}^{n},\hat{Z}_{k,j}^{n}],[\hat{Y}_{k,j}^{n},\hat{B}_{k,j}^{n}],[\hat{Z}_{k,j}^{n},\hat{B}_{k,j}^{n}])\to \mathcal{W}_{k,j},
	\end{equation*}
	in probability as $n$ goes to infinity, where
	\begin{align*}
		&[\hat{Y}_{k,j}](t)  =\int_{0}^{t}\lambda\frac{n_{k}(j)}{p_{j}}F^{c}(t-s)\bar{X}_{k}(s)\bar{Y}_{j}(s)ds,       \\
		&[\hat{Z}_{k,j}](t)  =\int_{0}^{t}\lambda\frac{n_{k}(j)}{p_{j}}F(t-s)\bar{X}_{k}(s)\bar{Y}_{j}(s)ds,           \\
		&[\hat{B}_{k,j}](t)  =\int_{0}^{t}\lambda\frac{n_{k}(j)}{p_{j}}\bar{Y}_{k}(s)\big(p_{j}-\bar{X}_{j}(s)\big)ds,\\
        &[\hat{Y}_{k,j},\hat{Z}_{k,j}](t) =-\int_{0}^{t}\lambda\frac{n_{k}(j)}{p_{j}}\bar{Y}_{k}(s)\big(p_{j}-\bar{X}_{j}(s)\big)ds,\\
        &[\hat{Y}_{k,j},\hat{B}_{k,j}](t)=[\hat{Z}_{k,j},\hat{B}_{k,j}](t) =0.
	\end{align*}
\end{lemma}
\begin{proof}
	We have
	\begin{equation*}
		[\hat{Y}_{k,j}](t)=\lim_{n\to\infty}[\hat{Y}_{k,j}^{n}](t)=\lim_{n\to\infty}\sum_{s\leq t}(\Delta\hat{Y}_{k,j}^{n}(s))^2=\lim_{n\to\infty}\frac{1}{\#V^{n}}\sum_{i=1}^{A^{n}_{k,j}(t)}\one
		_{\{t<\tau_{i}^{n,k,j}+\eta_i^k\}},
	\end{equation*}
    where $\Delta\hat{Y}^{n}_{k,j}(s)=\hat{Y}^{n}_{k,j}(s)-\hat{Y}^{n}_{k,j}(s-)$ represents the jump at time $s$ of $\hat{Y}^{n}_{k,j}$. As there is no jump of a stochastic integral with respect to the Lebesgue measure, and the sum is a counting process with jumps of size $(\#V^n)^{-1/2}$, we obtain the last equality. Applying the Functional Law of Large Numbers, Theorem \ref{FLLN}, the result follows. 

    We compute two other cases, when the processes involved are different. Recall that the jumps associated to $\hat{Y}^{n}_{k,j}$ occur at the same time as those associated to $\hat{Z}^{n}_{k,j}$ but with different orientation, that is, when $\hat{Y}^{n}_{k,j}$ decreases, $\hat{Z}^{n}_{k,j}$ increases. Then,

    \begin{align*}
		[\hat{Y}_{k,j},\hat{Z}_{k,j}](t)&=\lim_{n\to\infty}[\hat{Y}_{k,j}^{n},\hat{Z}_{k,j}^n](t)=\lim_{n\to\infty}\sum_{s\leq t}\Delta\hat{Y}_{k,j}^{n}(s)\Delta\hat{Z}_{k,j}^{n}(s)\\
        &=\lim_{n\to\infty}-\sum_{s\leq t}(\Delta\hat{Y}_{k,j}^{n}(s))^2=-\lim_{n\to\infty}\frac{1}{\#V^{n}}\sum_{i=1}^{A^{n}_{k,j}(t)}\one
		_{\{t<\tau_{i}^{n,k,j}+\eta_i^k\}}.
	\end{align*}
    
    Notice that $\hat{Y}^{n}_{k,j}$ and $\hat{B}^{n}_{k,j}$ do not share jumps with probability $1$, then their quadratic covariation is identically zero. The same applies for $\hat{Z}^{n}_{k,j}$ and $\hat{B}^{n}_{k,j}$.
    
    The proofs for the remaining cases are similar, with the obvious modifications. Observe that if the indices $(k_1,j_1)\neq(k_2,j_2)$, then automatically they share no jump and the autocovariance is zero.
\end{proof}

\begin{lemma}
	\label{LemmaConv}As $n\to\infty$,
	\begin{equation*}
		(\hat{Y}^{n}_{k,j},\hat{Z}^{n}_{k,j},\hat{B}^{n}_{k,j})_{k,j\in\{1,\dots,\ns\}}
		\Rightarrow(\hat{Y}_{k,j},\hat{Z}_{k,j},\hat{B}_{k,j})_{k,j\in\{1,\dots,\ns\}}
	\end{equation*}
	where $(\hat{Y}_{k,j},\hat{Z}_{k,j},\hat{B}_{k,j})_{k,j\in\{1,\dots,\ns\}}$
	is the zero-mean Gaussian vector with independent increments and covariances given in Table \ref{TableOfCov}.
\end{lemma}

\begin{proof}
	The convergence follows from Proposition \ref{Convquadvar} and Theorem $1
		.4$ (Chapter $7$) in \cite{Kurtz}. The covariances between the limiting
	processes are given in Table \ref{TableOfCov}. Observe that the independent increments condition implies that the covariances of the processes in different times, say $r$ and $t$, will be the covariance at the minimum $r\wedge t$. Indeed, assume $r\geq t$. Without loss of generality, we have
    \begin{align*}
        \E[\hat{Y}_{k,j}(t)\hat{Z}_{k,j}(r)]&=\E[\hat{Y}_{k,j}(t)\hat{Z}_{k,j}(r)-\hat{Y}_{k,j}(t)\hat{Z}_{k,j}(t)+\hat{Y}_{k,j}(t)\hat{Z}_{k,j}(t)]\\
        &=\E[\hat{Y}_{k,j}(t)(\hat{Z}_{k,j}(r)-\hat{Z}_{k,j}(t))]+\E[\hat{Y}_{k,j}(t)\hat{Z}_{k,j}(t)]\\
        &=[\hat{Y}_{k,j},\hat{Z}_{k,j}](t)\\
        &=-\int_0^t\lambda c_{k,j}F^{c}(t-s)\bar{X}_k(s)\bar{Y}_j(s)ds.
    \end{align*}
\end{proof}

As we characterized the Gaussian noises, we state the last ingredient to
prove the result. Observe that $(\hat{X}^{n}_{1},\hat{Y}^{n}_{1},\hat{Z}^{n}_{1},\dots,\hat{X}^{n}_{\ns},\hat{Y}^{n}_{\ns},\hat{Z}^{n}_{\ns})(t)$ solves a system of equations in $D^{3\ns}$ of type
\begin{align*}
	\phi_1^X(t)=&-\phi_1^Y(t)-\phi_1^Z(t),\\
	\phi_1^Y(t)=&\,y_{1}^{g}(t)+\sum_{j=1}^{\ns}c_{j1}\int_0^tF^c (t-s)(d_{11}\phi_1^X(s)y_{j}(s)+d_{12}f_1(s)\phi_{j}^{Y}(s))ds\\
	&-\sum_{j=1}^{\ns}c_{j1}\int_0^t\phi_1^Y(s)x_j(s)+g_1(s)\phi_j^Y(s)ds,\\
	\phi_1^Z(t)=&\,z_{1}^{g}(t)+\sum_{j=1}^{\ns}c_{j1}\int_0^tF(t-s)(d_{11}\phi_1^X(s)y_{j}(s)+d_{12}f_1(s)\phi_{j}^{Y}(s))ds\\
	&+\sum_{j=1}^{\ns}c_{j1}\int_0^t\phi_1(s)x_j(s)+g_1(s)\phi_j^Y(s)ds,\\
	&\vdots\\
	\phi_{\ns}^X(t)=&-\phi_{\ns}^Y(t)-\phi_{\ns}^Z(t),\\
	\phi_{\ns}^Y(t)=&\,y_{\ns}^{g}(t)+\sum_{j=1}^{\ns}c_{j\ns}\int_0^tF^c(t-s)(d_{\ns1}\phi_{\ns}^X(s)y_{j}(s)+d_{\ns2}f_{\ns}(s)\phi_{j}^{Y}(s))ds\\
	&-\sum_{j=1}^{\ns}c_{j\ns}\int_0^t\phi_{\ns}^Y(s)x_j(s)+g_1(s)\phi_j^Y(s)ds,\\
	\phi_{\ns}^Z(t)=&\,z_{\ns}^{g}(t)+\sum_{j=1}^{\ns}c_{j\ns}\int_0^tF(t-s)(d_{11}\phi_{\ns}^X(s)y_{j}(s)+d_{\ns2}f_{\ns}(s)\phi_{j}^{Y}(s))ds\\
	&+\sum_{j=1}^{\ns}c_{j\ns}\int_0^t\phi_{\ns}(s)x_j(s)+g_{\ns}(s)\phi_j^Y(s)ds.
\end{align*}

The system can be seen as a map $\Gamma:D^{4\ns}\to D^{3\ns}$ with
\begin{equation*}
	(x_{1},y_{1},y_{1}^{g},z_{1}^{g},\dots,x_{\ns},y_{\ns},y_{\ns}^{g},z_{\ns}
	^{g})\mapsto(\phi_{1}^{X},\phi_{1}^{Y},\phi_{1}^{Z},\dots,\phi_{\ns}^{X},
	\phi_{\ns}^{Y},\phi_{\ns}^{Z}).
\end{equation*}

The following lemma is analogue to Lemma $3.14$ in \cite{CL25}.

\begin{lemma}
	\label{Volterraintegral} Assume that $f_{1},g_{1},\dots,f_{\ns},g_{\ns}$
	are continuous. Then the map $\Gamma:D^{4\ns}\to D^{3\ns}$ has a unique
	solution $(\phi_{1}^{X},\dots,\phi_{\ns}^{Z})\in D^{3\ns}$ to the
	integral system and is continuous in the Skorohod topology, in the sense
	that, if
	\begin{equation*}
		(x_{1}^{n},y_{1}^{n},y_{1}^{g,n},z_{1}^{g,n},\dots,x_{\ns}^{n},y_{\ns}
		^{n},y_{\ns}^{g,n},z_{\ns}^{g,n})\to(x_{1},y_{1},y_{1}^{g},z_{1}^{g},
		\dots,x_{\ns},y_{\ns},y_{\ns}^{g},z_{\ns}^{g}),
	\end{equation*}
	with $(x_{1},y_{1},\dots,x_{\ns},y_{\ns})\in C^{2\ns}$, then
	\begin{equation*}
		\Gamma(x_{1}^{n},y_{1}^{n},y_{1}^{g,n},z_{1}^{g,n},\dots,x_{\ns}^{n},
		y_{\ns}^{n},y_{\ns}^{g,n},z_{\ns}^{g,n})\to\Gamma(x_{1},y_{1},y_{1}^{g}
		,z_{1}^{g},\dots,x_{\ns},y_{\ns},y_{\ns}^{g},z_{\ns}^{g}),
	\end{equation*}
	in $D^{3}$ as $n\to\infty$.
\end{lemma}

Applying the established convergence of the noises in Lemmas \ref{Convzerohat} and \ref{LemmaConv} to Lemma \ref{Volterraintegral}, we obtain the desired result.

The remaining cases in Lemma \ref{lemma:cov} follow by the pairwise independence between the random variables $\{\eta_i^0\}_{i\in\bbN}$ and $\{\eta_i\}_{i\in\bbN}$.

\subsection{The infinite quasi-transitive sub-exponential case}

In this section, we extend the previous results to $\Gamma$, an infinite, connected, locally finite, quasi-transitive graph with subexponential growth. Our strategy proceeds in four steps. First, we analyze the dynamic inside large finite boxes, splitting each box into a main interior region and a thin boundary layer whose volume is negligible compared to the whole box. The second step is to follow the process only up to the first time that the interior region comes into contact with this boundary layer. In the third step, we couple the dynamics to a first-passage percolation model and use the Shape Theorem to show that this boundary-hit time diverges as the boxes grow. Finally, we transfer the finite-box estimates to the infinite graph by combining this divergence of the aforementioned boundary-hit times with the analog of the first-step analysis for the limiting regime.

In order to prove these results, we need some definitions. We call a path $v=v_0,v_1,\ldots,v_m=u$ \emph{open} if there exist epochs
$\theta_i\in N_{v_{i-1},v_i}$, $1\leq i\leq m$, such that
\[
  \theta_1\le \theta_2\le \cdots \le \theta_m.
\]
For $t\ge 0$, we say that a path is \emph{$t$-open} if it is open and $\theta_1\ge t$, and its
\emph{terminal epoch} is $\theta_m$.

We define the \emph{$t$-passage time} from $v$ to $u$ as
\[
  T_t(v,u)
  :=\inf\Bigl\{\theta_m-t:\ \text{there exists a $t$-open path from $v$ to $u$
  with terminal epoch }\theta_m\Bigr\}\in[0,\infty],
\]
with the convention $\inf\varnothing=\infty$. Any $t$-open path that attains this
infimum (when it exists) is called a \emph{$t$-geodesic}. We write
$T(v,u):=T_0(v,u)$ and call it the first passage time from $v$ to $u$.

\begin{lemma}  \label{Growthlemma}
	Let $f(n)$ be the growth function of a Cayley graph with subexponential growth. 	Then there exists a function $g:\mathbb{N}\to\mathbb{N}$ such that:
	\begin{enumerate}[before=\leavevmode,
                  label=\textup{(\roman*)},
                  ref=\textup{\roman*}]
		\item \label{itm:g_infty} $\lim\limits_{n\to\infty} g(n)=\infty$;
		\item \label{itm:g_bound} $0<g(n)<n$, for sufficiently large $n$;
		\item \label{itm:g_limit} $\displaystyle \lim\limits_{n\to\infty}\frac{f(n)-f(n-g(n))}{f(n)}=0$.
	\end{enumerate}
\end{lemma}

\begin{proof}
	The growth function $f(n)$ of an infinite Cayley graph with subexponential growth satisfies the following properties, which form the basis of the proof:
	\begin{enumerate}[before=\leavevmode,
                  label=\textup{(S\arabic*)},
                  ref=\textup{S\arabic*}]
		\item \label{prop:monotonicity} {Monotonicity:} The function $f(n)$ is non-decreasing and unbounded, since the graph is infinite and connected. In particular, $f(n) \ge f(m)$ for all $n \ge m$.
		\item \label{prop:folner} {Convergence of the Relative Growth Rate:} A fundamental result in geometric group theory is that a finitely generated group has subexponential growth then the sequence of its balls $(B_n)$ is a Følner sequence \cite{garrido2013introduction}. This implies that the boundary of the balls is asymptotically small relative to their volume. As a direct consequence, the relative growth rate
  
	\begin{equation*}
		\Delta(n):=\frac{f(n)-f(n-1)}{f(n-1)}
	\end{equation*}
        converges to zero:
		      \[
			      \lim_{n\to\infty} \Delta(n) = \lim_{n\to\infty} \frac{f(n)-f(n-1)}{f(n-1)} = 0.
		      \]
		      This is the crucial property we will use.
	\end{enumerate}

	\textit{ Construction of $g(n)$.}
	To construct a function $g(n)$ with the desired properties, we first define a tail function $M(N)$ that measures the supremum of the growth rate from an index $N$ onward:
	\begin{equation}
		M(N):=\sup_{m\ge N}\Delta(m),\qquad N\in\mathbb{N}. \tag{Def. $M$} \label{def:M}
	\end{equation}
	By property \ref{prop:folner}, we know that $M(N)$ is non-increasing and $M(N)\downarrow 0$ as $N\to\infty$.

	Now, for each $n\ge 2$, we define $g(n)$ adaptively, based on the convergence speed of $M(N)$:
	\begin{equation}
		g(n):=\min\!\left\{\left\lfloor \frac{n}{2}\right\rfloor,\ \left\lfloor\, M\!\Big(\left\lfloor\frac{n}{2}\right\rfloor\Big)^{-\frac12}\right\rfloor\right\}. \tag{Def. $g$} \label{def:g}
	\end{equation}
	We verify the first two conditions of Lemma \ref{Growthlemma}. Clearly $0<g(n)\le \lfloor n/2\rfloor<n$, which proves condition \eqref{itm:g_bound}. Since $\lim_{n\to\infty} \lfloor n/2 \rfloor = \infty$, we have that $\lim_{n\to\infty} M(\lfloor n/2 \rfloor) = 0$, which implies $\lim_{n\to\infty} \lfloor M(\lfloor n/2 \rfloor)^{-1/2} \rfloor = \infty$. Since $g(n)$ is the minimum of two sequences that diverge to infinity, it also diverges to infinity, proving condition \eqref{itm:g_infty}.

	\textit{Verification of the Convergence Condition.}
	To prove condition \eqref{itm:g_limit}, we express the difference as a sum of increments:
	\[
		f(n)-f(n-g(n))=\sum_{k=0}^{g(n)-1}\big(f(n-k)-f(n-k-1)\big).
	\]
	Dividing by $f(n)$, we obtain the expression we want to bound:
	\[
		\frac{f(n)-f(n-g(n))}{f(n)} = \sum_{k=0}^{g(n)-1}\frac{f(n-k)-f(n-k-1)}{f(n)}.
	\]
	Since $k \ge 0$, the monotonicity property \ref{prop:monotonicity} ensures that $f(n) \ge f(n-k-1)$. Using this smaller denominator in place of $f(n)$ gives us an upper bound:
	\[
		\frac{f(n)-f(n-g(n))}{f(n)} \le \sum_{k=0}^{g(n)-1}\frac{f(n-k)-f(n-k-1)}{f(n-k-1)} = \sum_{k=0}^{g(n)-1}\Delta(n-k).
	\]
	For each term $\Delta(n-k)$ in the sum, the index $m = n-k$ satisfies $n-g(n)+1 \le m \le n$. Since $g(n)\le \lfloor n/2\rfloor$, the smallest index is $m \ge n-(g(n)-1) \ge n-(\lfloor n/2\rfloor-1) \ge \lceil n/2 \rceil$. Thus, $m \ge \lfloor n/2 \rfloor$ for all $k$ in the sum. By the definition of $M(N)$ in \eqref{def:M}:
	\[
		\Delta(n-k)\ \le\ M\!\Big(\Big\lfloor\frac{n}{2}\Big\rfloor\Big)
		\quad\text{for all }k=0,\dots,g(n)-1.
	\]
	Substituting this uniform bound into the sum, which has $g(n)$ terms:
	\[
		\frac{f(n)-f(n-g(n))}{f(n)} \le g(n)\cdot M\!\Big(\Big\lfloor\frac{n}{2}\Big\rfloor\Big).
	\]
	Finally, we use the definition of $g(n)$ from \eqref{def:g}. Since $g(n)$ is the minimum of two terms, it is less than or equal to each of them. Using the second term:
	\[
		g(n)\cdot M\!\Big(\Big\lfloor\frac{n}{2}\Big\rfloor\Big) \le \left\lfloor M\!\Big(\Big\lfloor\frac{n}{2}\Big\rfloor\Big)^{-\frac12}\right\rfloor \cdot M\!\Big(\Big\lfloor\frac{n}{2}\Big\rfloor\Big).
	\]
	For $x>0$, the inequality $\lfloor x^{-1/2}\rfloor \cdot x \le x^{-1/2} \cdot x = x^{1/2}$ holds. With $x=M(\lfloor n/2\rfloor)$, we obtain:
	\[
		\frac{f(n)-f(n-g(n))}{f(n)}\ \le\ M\!\Big(\Big\lfloor\frac{n}{2}\Big\rfloor\Big)^{1/2}.
	\]
	Since $M(N)\to 0$ as $N\to\infty$, the right-hand side tends to 0 as $n\to\infty$.
\end{proof}

Now fix $v_{0}\in V$ and, for $r>0$, let $B_{r}:=B(v_{0},r)$, and let $f$ and $g$ be as in Lemma~\ref{Growthlemma}. We consider the Maki--Thompson model with spontaneous stifling restricted to the region $B_{n}$, and we run the dynamics only up to a time strictly before boundary effects can influence the inner ball
$B_{n-g(n)}$.

In order to do this, let the first time when there is an open path between the ball inside, $B_{n-g(n)}$, and the boundary of the entire region, $\partial B_{n}$, and consider its half (see
Figure \ref{figboxes}):
\begin{equation*}
	\varphi_{n}:=\frac{1}{2}\min_{v\in B_{n-g(n)},u\in\partial B_n}T(v,u).
\end{equation*}

\begin{figure}[h]
	\centering
	\begin{tikzpicture}[]
		\draw[thick] (0,0) rectangle (4,4);
		\draw[thick] (0.2,0.2) rectangle (3.8,3.8);

		\fill[pattern=dots, pattern color=gray!70]
		(0.2,0.2) rectangle (3.8,3.8);

		\node at (3.8,-0.2) {$B_{n}$};
		\node at (2.9,0.8) {$B_{n-g(n)}$};
	\end{tikzpicture}
	 \caption{Schematic representation of the region $B_{n}$ and its inner region $B_{n-g(n)}$.}

	\label{figboxes}
\end{figure}
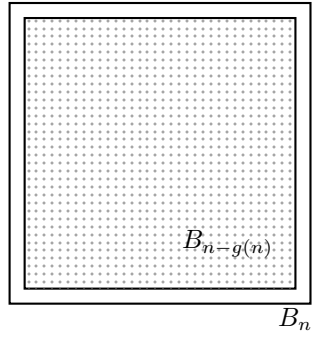

Let $\tau:V\to\{1,\dots,\ns\}$ be the type function
and define $\tau_{n}:V^{n}\to\{0,1,\dots,\ns\}\cup\{\emptyset\}$ as the map
\begin{equation*}
	\tau_{n}(v)=
    \begin{cases}
  \tau(v),   & \text{if } v\in B_{n-g(n)},\\[0.2em]
  \emptyset, & \text{otherwise}.
\end{cases}
\end{equation*}

As in the first case, we write $V^{n}_{k}:=V^{n}\cap\tau^{-1}(k)$.
The proportion of vertices of type $k$ that are free of boundary
effects is then
\[
  \rho^{n}_{k}
  := \frac{\#\bigl(V^{n}_{k}\setminus\tau_{n}^{-1}(\emptyset)\bigr)}{\#V^{n}_{k}}.
\]

By the definition of $\varphi_{n}$, up to time $\varphi_{n}$ the dynamics
inside $B_{n-g(n)}$ does not interact with $\partial B_{n}$ and thus behaves as if it were defined on a quasi-transitive graph.

As a corollary of Lemma \ref{stocintv}, for $k\in\{1,\dots,\ns\}$ and $t\leq\varphi_{n}$, the $\F_{t}^{n}$-stochastic intensity of $A^{n,k}$ lies in the interval
\begin{equation}
	\label{stocintinterval}\Bigg[\lambda\rho_{k}^{n}X_{k}^{n}(t)\sum_{j=1}^{\ns}
	n_{k}(j)\frac{Y_{j}^{n}(t)}{\#V_{j}^{n}}-(1-\rho_{k}^{n})\#(V^{n}_{k}\setminus
	B_{n}), \ \lambda X_{k}^{n}(t)\sum_{j=1}^{\ns}n_{k}(j)\frac{Y_{j}^{n}(t)}{\#V_{j}^{n}}
	+(1-\rho_{k}^{n})\#(V^{n}_{k}\setminus B_{n})\Bigg].
\end{equation}

We then have the following lemma.

\begin{lemma}
  \label{stocintgeneral}
  For each $k\in\{1,\dots,\ns\}$ and $t\le \varphi_{n}$, the
  $\F_{t}^{n}$-stochastic intensity of $A^{n,k}$ is
  \begin{equation*}
    \lambda\,X_{k}^{n}(t)\sum_{j=1}^{n} n_{k}(j)\,\frac{Y_{j}^{n}(t)}{\#V_{j}^{n}}
    + o(\#V^{n}).
  \end{equation*}
 
\end{lemma}

\begin{proof}
	In order to apply Lemma \ref{Growthlemma}, recall that $f(n)=B_n=\#V^n$. Notice that 
    \begin{equation*}
		\rho^{n}_{k}=1-\frac{\#V^{n}_{k}\cap\tau_{n}^{-1}(\emptyset)}{\#V^{n}_{k}},
	\end{equation*}
    which, by condition \eqref{itm:g_limit} in Lemma \ref{Growthlemma}, converges to $1$ as $n$ goes to infinity. Combining this with Equation \eqref{stocintinterval}, item \eqref{itm:g_limit} in Lemma \ref{Growthlemma} and $t\leq\varphi_n$, the result follows.
    
\end{proof}

\subsubsection{Main results}

The proofs of the Functional Law of Large Numbers and the Central Limit Theorem in the general case are derived from the results stated below.

\begin{lemma}
	For any $k\in\{1,\dots,\ns\}$, the sequences of processes
	$\{A^{n,k}/\#V^{n}\}$ and $\{B^{n,k}/\#V^{n}\}$ are tight.
\end{lemma}
\begin{proof}
	By Lemma \ref{stocintgeneral}, the stochastic intensities of the processes are almost surely bounded for any time $t$. The result follows as it is sufficient to verify Aldous' tightness criterion, as in Lemma $3.1$ in \cite{CL25}.
\end{proof}

Now we observe that Lemma \ref{convinit} still holds without any adjustment for
this case and all the next results are immediate consequences of Lemmas \ref{Growthlemma}
and \ref{stocintgeneral}. We now state  the Shape Theorem, a pivotal result that ensures that our limit quantities are defined for any time $t\geq0$.

\begin{theorem}[Shape Theorem for FPP on Graphs of Subexponential Growth ({\cite[Corollary 2.3]{benjamini2015first})}]\label{ShapeFPP}
	Let $\Gamma = (V,E)$ be an infinite, quasi-transitive graph. Consider first-passage percolation (FPP) on $\Gamma$ by assigning to each edge $e \in E$ an i.i.d.  Poisson process of rate  $\lambda$. 

	If $\Gamma$ has subexponential growth, then for almost every realization $\omega$, the following holds: for any $\epsilon > 0$, there exists an $n_0 = n_0(\omega, \epsilon)$ such that for all $n \ge n_0$,
	\[
		B(o, \lambda (1-\epsilon)n) \subset B_\omega(o, n) \subset B(o, \lambda (1+\epsilon)n).
	\]
\end{theorem}

Now we show that $\varphi_{n}$ is an explosive
sequence, that is, $\varphi_{n}\uparrow\infty$ as $n$ goes to infinity. Recall the invariance of the Poisson clocks on edges and notice that $T$ defines the random metric that generates the balls $B_{\omega}(0,n)$. Moreover, Theorem \ref{ShapeFPP} asserts that almost surely, for any $\varepsilon>0$, there is an $n_0$ such that for any such $u$ and $v$, $T(u,v)\geq\lambda(1-\varepsilon)(n-g(n))$ for any $n\geq n_0$. Since
we have finite vertices in $B_{n-g(n)}$, there is an $\bar{n}_0$ such that for $n\geq\bar{n}_0$,
\begin{equation*}
    \varphi_{n}=\frac{1}{2}\min_{v\in B_{n-g(n)},u\in\partial B_n}T(v,u)=\frac{1}{2}\min_{v\in \partial B_{n-g(n)},u\in\partial B_n}T(v,u)\geq\frac{\lambda}{2}(1-\varepsilon)g(n)
\end{equation*}
 almost surely and we conclude the claim.

We recall that $\#\tau_{n}^{-1}(\emptyset)/\#V^{n}$ goes to zero as $n$ goes to infinity, and it suffices to show that the boundary-affected quantities $X_{\emptyset}^{n},Y_{\emptyset}^{n},Y_{\emptyset}^{n}\in o(\#V^{n})$ therefore $\bar{X}_{\emptyset}=\bar{Y}_{\emptyset}=\bar{Z}_{\emptyset}\equiv0$.

We recall that the construction of the Functional Central Limit Theorem will follow the same path as the construction in the first case up to time $\varphi_{n}$. Notice that the noises $\hat{Y}^{n}_{k,j},\hat{Z}^{n}_{k,j}$ and $\hat{B}^{n}_{k,j}$, $k,j\in\{1,\dots,\ns\}$ are defined as the differences between the counting processes and their actual stochastic intensities. Nevertheless, the results will hold by Lemma \ref{stocintgeneral}.

Observe that we cannot control the fluctuations of type $\emptyset$ as the related quantities are nonnegative.

\section{Simulations}

\paragraph{Event–driven simulator (Maki–Thompson with spontaneous stifling).}
We simulate the Maki–Thompson rumor dynamics with spontaneous stifling, with states $(X,Y,Z)$, on a fixed graph $G=(V,E)$ using a continuous–time, event–driven Gillespie scheme.

Initial proportions for $(X,Y,Z)$ are chosen and vertex states are sampled accordingly; vertices starting in state $Y$ each draw and schedule their own stifling time $\eta\sim F$.

\begin{itemize}
    \item \textbf{Spontaneous stifling:}  upon \emph{entering} state $Y$, an
individual draws an independent stifling time $\eta\sim F$, scheduling its future $Y\!\to\!Z$ transition.

    \item \textbf{Contact process:} all $X\!-\!Y$, $Y\!-\!Y$, and $Y\!-\!Z$ edges generate contacts, each with per-edge rate $\lambda$. 
    Let $n_{XY}(t)$, $n_{YY}(t)$, and $n_{YZ}(t)$ be their counts and set $M(t)=n_{XY}(t)+n_{YY}(t)+n_{YZ}(t)$.
    In the Gillespie scheme, contact events arrive with total rate $\lambda\,M(t)$.

\end{itemize}

At each step, the next event time is the minimum between the next scheduled spontaneous stifling and the next contact time. 
If a \textbf{spontaneous stifling} occurs, we update $Y\!\to\!Z$ for the corresponding vertex.
If a \textbf{contact} occurs, we choose the edge category by thinning, proportionally to its rate contribution ($\lambda n_{XY}$, $\lambda n_{YY}$, $\lambda n_{YZ}$), then pick one edge \emph{uniformly} within that category, and update:
\[
X\!-\!Y:\ X\to Y\ (\text{draw and schedule }\eta),\quad
Y\!-\!Y:\ Y\to Z\ ,\quad
Y\!-\!Z:\ Y\to Z.
\]

We record $(X,Y,Z)$ on a uniform time grid using left–continuous interpolation between events and report averages over independent replicas up to $t_{\max}$.

\paragraph{Numerical solution of the integral equations.}
The deterministic orbit-averaged limits $(\bar X_k,\bar Y_k,\bar Z_k)$ solve the integral equations in Theorem~\ref{FLLN}. We discretize $t_n=n\Delta t$ and evaluate each convolution by the trapezoidal rule  yielding a global error of order $O(\Delta t^2)$.

\subsection{Quasi–transitive five–orbit graph.}
The quasi–transitive graph in Fig.~\ref{fig:grafo5}, used in Fig.~\ref{fig:orbit-compare}, comprises five
vertex orbits with inter–orbit neighbor counts
\[
N=\begin{pmatrix}
0&24&2&0&3\\
24&0&2&0&1\\
2&2&0&4&0\\
0&0&4&0&0\\
3&1&0&0&0
\end{pmatrix}.
\]
Thus, a type–0 vertex (degree~29) connects to $24$ vertices of type~1, $2$ of
type~2, and $3$ of type~4; type~1 (degree~27) connects to $24$ of type~0, $2$ of
type~2, and $1$ of type~4; type~2 (degree~8) connects to $2$ of type~0, $2$ of
type~1, and $4$ of type~3; type~3 (degree~4a) connects only to type~2 (four
edges); and type~4 (degree~4b) connects only to types~0 and~1 (three and one
edges, respectively). The coefficients $n_{kj}$ (inter–orbit neighbor counts) and the orbit masses $p_j$ 
jointly encode the graph’s structural heterogeneity and
govern both the orbit–specific dynamics seen in Fig.~\ref{fig:orbit-compare} and the distributional shifts in Fig.~\ref{fig:vol-dists}.

\begin{figure}[h!]
    \centering
    \includegraphics[width=0.5\linewidth]{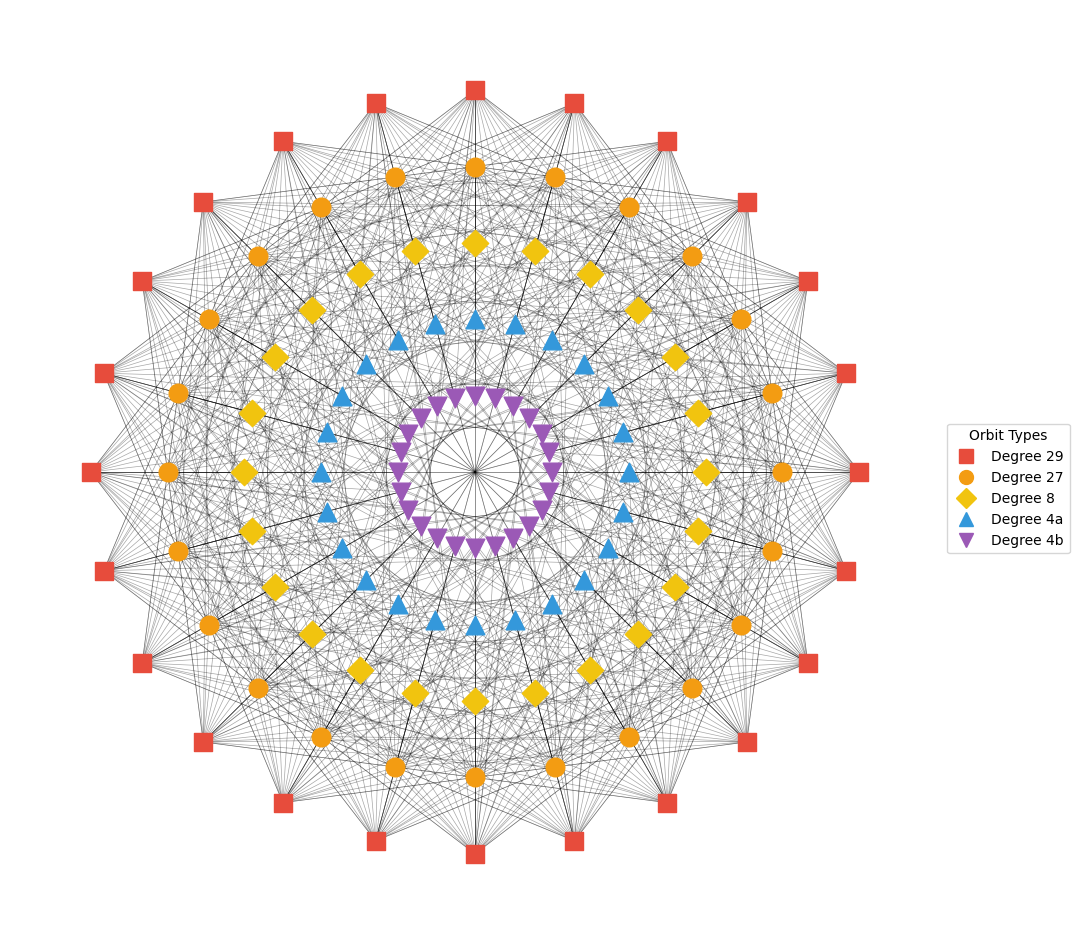}
\caption{Five–orbit quasi–transitive graph.}  
\label{fig:grafo5}
\end{figure}

\paragraph{Spreaders and ignorants on the five–orbit graph (Fig.~\ref{fig:orbit-compare})}
Panels (a) and (b) display spreaders and panels (c) and (d) display ignorants, with simulations on the left and the solution of the correspondent integral equations on the right. High–degree orbits (29, 27) ignite earlier and reach higher peaks, but also quench faster due to the elevated frequency of spreader–spreader and spreader–stifler contacts; low–degree orbits (8, 4a, 4b) peak later and lower. The deterministic curves reproduce these timings and amplitudes, and the small residual gaps are consistent with finite–size noise ($O(1/\sqrt{\#V^n})$) plus modest quadrature discretization error. For ignorants, the decay is monotone across all orbits; low–degree orbits retain the largest ignorant mass at stationarity, which is also captured by the integral equations plateau. The two vertex types sharing degree 4 (types 4a and 4b) exhibit qualitatively different spreader dynamics. This comparison highlights that network topology, not just vertex degree, drives the propagation.

\paragraph{Effect of the stifling–time law (Fig.~\ref{fig:vol-dists})}

Replacing the stifling–time law $F$ (e.g., Weibull vs.\ truncated Cauchy)
produces the expected timing shifts: heavier tails postpone stifling events, sustain a larger spreader mass $\bar Y$ for longer, and translate the stifler trajectory $\bar Z$ rightward, with corresponding adjustments in peak timing and height.

\begin{figure}[!h]
  \centering
  \begin{subfigure}[b]{0.48\textwidth}
    \includegraphics[trim={0 0 0 0.74cm},clip,height=0.65\linewidth,width=\linewidth]{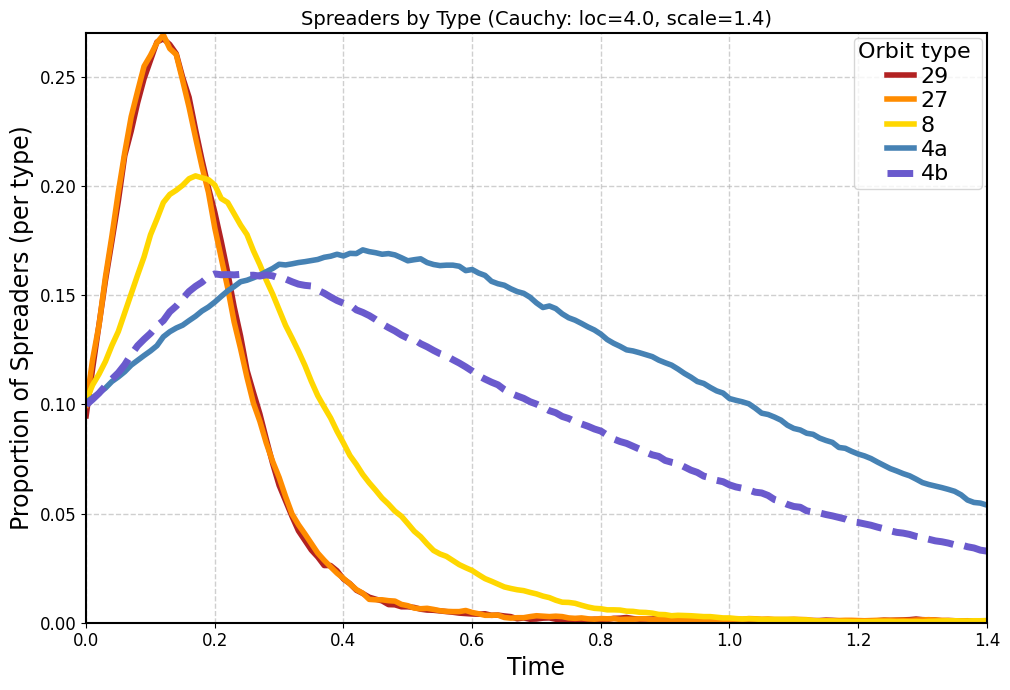}
    \caption{Spreaders by orbit (simulation).}
  \end{subfigure}\hfill
  \begin{subfigure}[b]{0.48\textwidth}
    \includegraphics[trim={0 0 0 1.6cm},clip,height=0.65\linewidth,width=\linewidth]{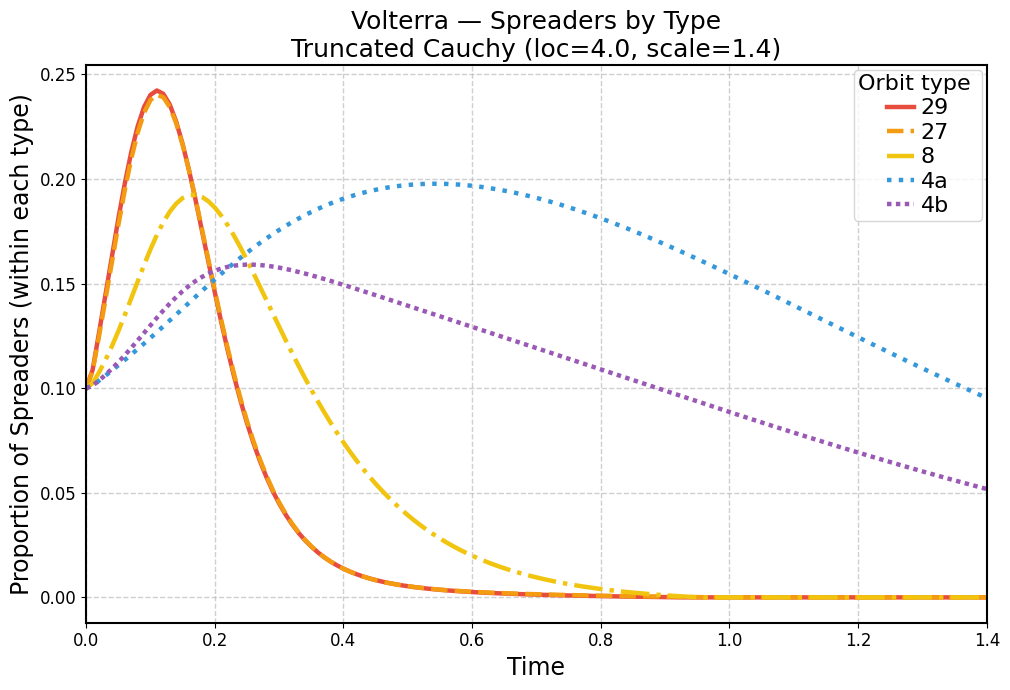}
    \caption{Spreaders by orbit (Integral equations).}
  \end{subfigure}

  \vspace{0.5em}

  \begin{subfigure}[b]{0.48\textwidth}
    \includegraphics[trim={0 0 0 0.74cm},clip,height=0.65\linewidth,width=\linewidth]{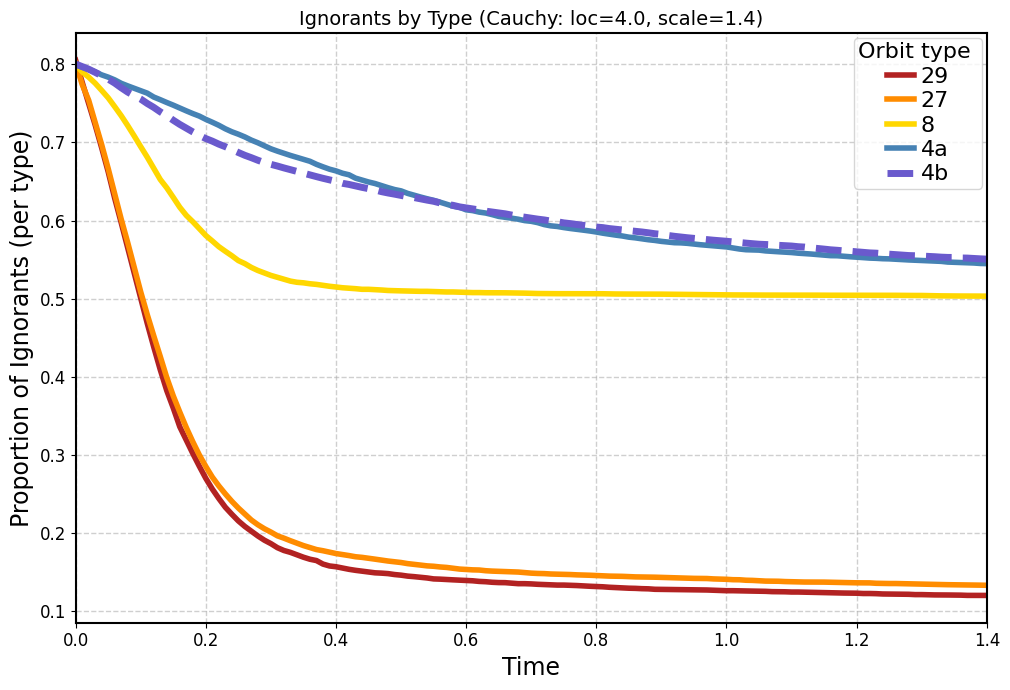}
    \caption{Ignorants by orbit (simulation).}
  \end{subfigure}\hfill
  \begin{subfigure}[b]{0.48\textwidth}
    \includegraphics[trim={0 0 0 1.6cm},clip,height=0.65\linewidth,width=\linewidth]{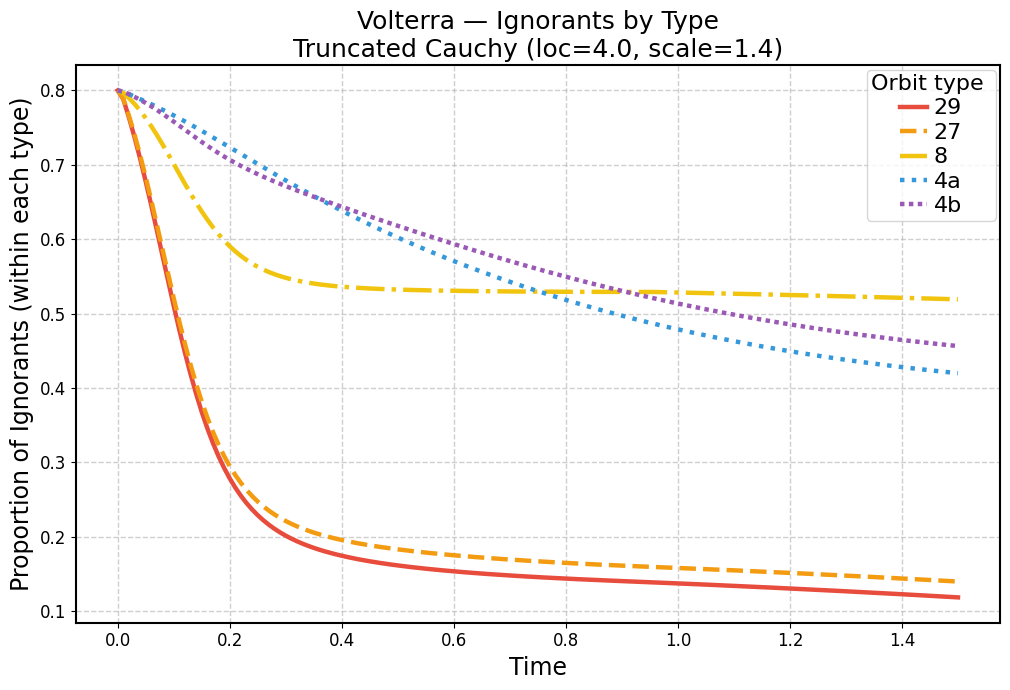}
    \caption{Ignorants by orbit (Integral equations).}
  \end{subfigure}

  \caption[Simulation vs. solutions of the integral equations on a quasi-transitive graph]{Simulation (30 independent runs) vs.\ solutions of the integral equations  on a quasi-transitive graph with five orbits (degrees $29,27,8,4\mathrm{a},4\mathrm{b}$). Stifling times $\sim$ truncated Cauchy (loc $=4.0$, scale $=1.4$). Curves show the proportion within each orbit, with simulations averaged over 30 runs.}

  \label{fig:orbit-compare}
\end{figure}

\begin{figure}[!h]
  \centering
  \includegraphics[trim={0 0 0 1.6cm},clip,width=0.6\linewidth]{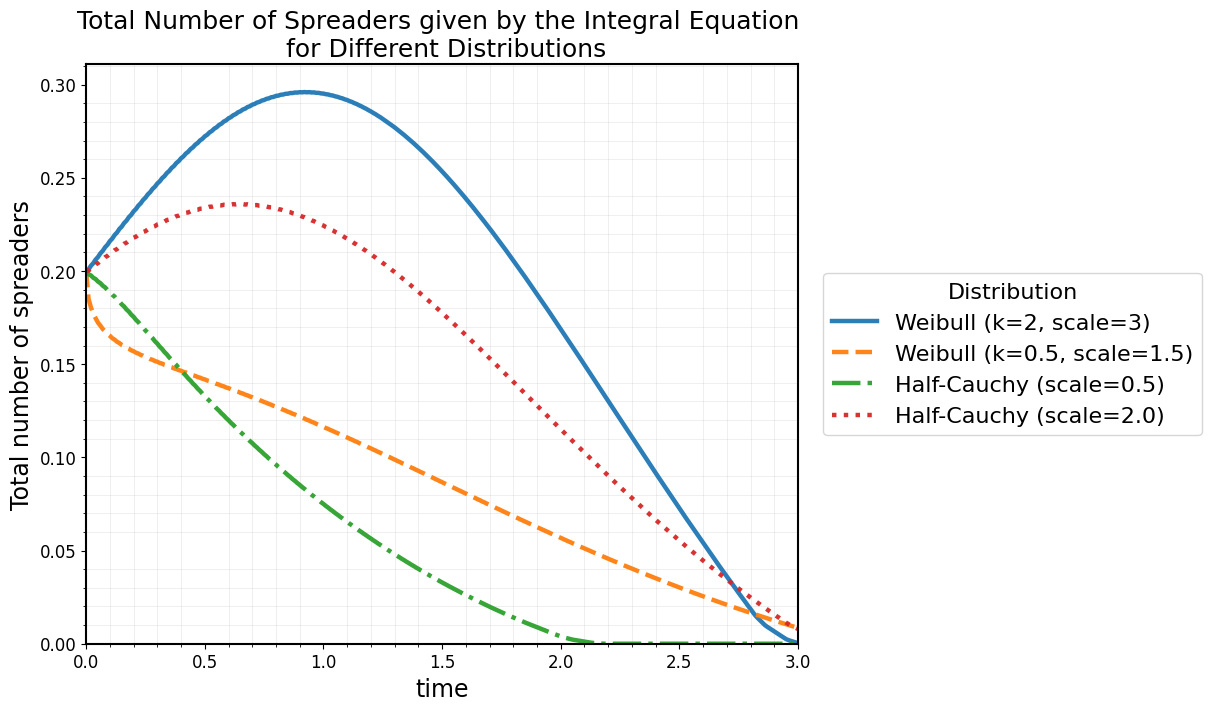}
  \caption[Solutions of the integral equations  under different waiting-time laws]
  {solutions of the integral equations  predictions for the spreader fraction under different stifling-time
   distributions (parameters as indicated in the legend).}
  \label{fig:vol-dists}
\end{figure}

\subsection{Square lattice $\mathbb{Z}^2$}

\paragraph{Single run and many runs (Fig.~\ref{fig:grid-compare})}
In panel~(a) it can be seen that a single stochastic trajectory (dashed) fluctuates around the
deterministic limit (solid), illustrating the LLN.  Panel~(b) overlays many
runs for two sizes.  The cloud for the larger grid is visibly tighter, in line
with $\mathrm{Var}[\text{averages}]=O(1/\#V^n)$. 

\paragraph{Fluctuations (Fig.~\ref{fig:fluctuations})}
Panel~(a) presents empirically centered, rescaled fluctuations $\sqrt{\#V^n}\,\bigl(\bar{Z}^n(t)-\bar Z_{\mathrm{emp}}(t)\bigr)$ for two sizes of the grid:
after the $\sqrt{\#V^n}$ scaling, the bands have comparable spread across $L$, in
accordance with the functional CLT. Panel~(b) shows histograms at a fixed time $t=2$:
the bell–shaped profiles and the normal fits support Gaussian
fluctuations around the limit of the integral equations.

\begin{figure}[!h]
  \centering
  \begin{subfigure}[t]{0.48\textwidth}
    \includegraphics[trim={0 0 0 1.6cm},clip,height=0.65\linewidth,width=\linewidth]{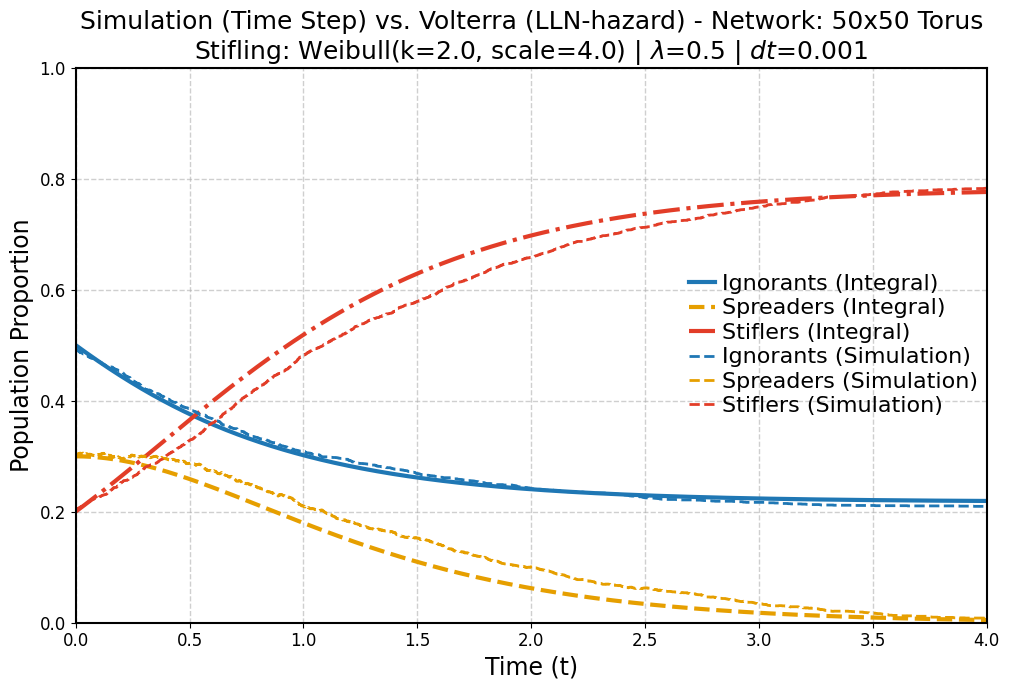}
    \caption{Single stochastic simulation on $\mathbb{Z}^2$ ($L=50$) vs.\ the deterministic integral equation solution.}

  \end{subfigure}\hfill
  \begin{subfigure}[t]{0.48\textwidth}
    \includegraphics[trim={0 0 0 0cm},clip,height=0.65\linewidth,width=\linewidth]{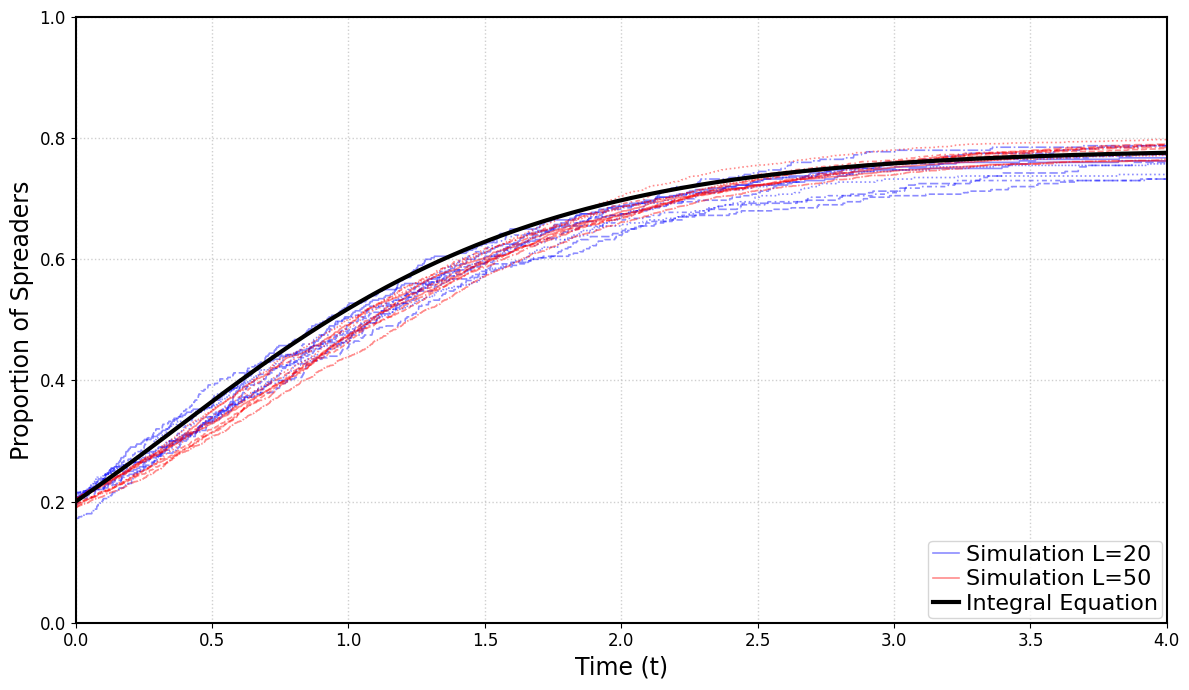}
    
    \caption{Comparison of 15 stochastic runs on $\mathbb{Z}^2$ with grid sizes $L=20$ and $L=50$ against the integral equations prediction. Individual simulations are shown as semi-transparent colored trajectories; the black curve is the deterministic integral equation solution.}

  \end{subfigure}
  \caption[Weibull case on $\mathbb{Z}^2$]
  {Weibull stifling times with shape $k=2$ and scale $=5$; contact rate $\lambda=0.5$.
   Solutions of the integral equations  overlaid for reference.}
  \label{fig:grid-compare}
\end{figure}

\begin{figure}[!h]
  \centering
  \begin{subfigure}[b]{0.48\textwidth}
    \includegraphics[trim={0 0 0 0.7cm},clip,width=\linewidth]{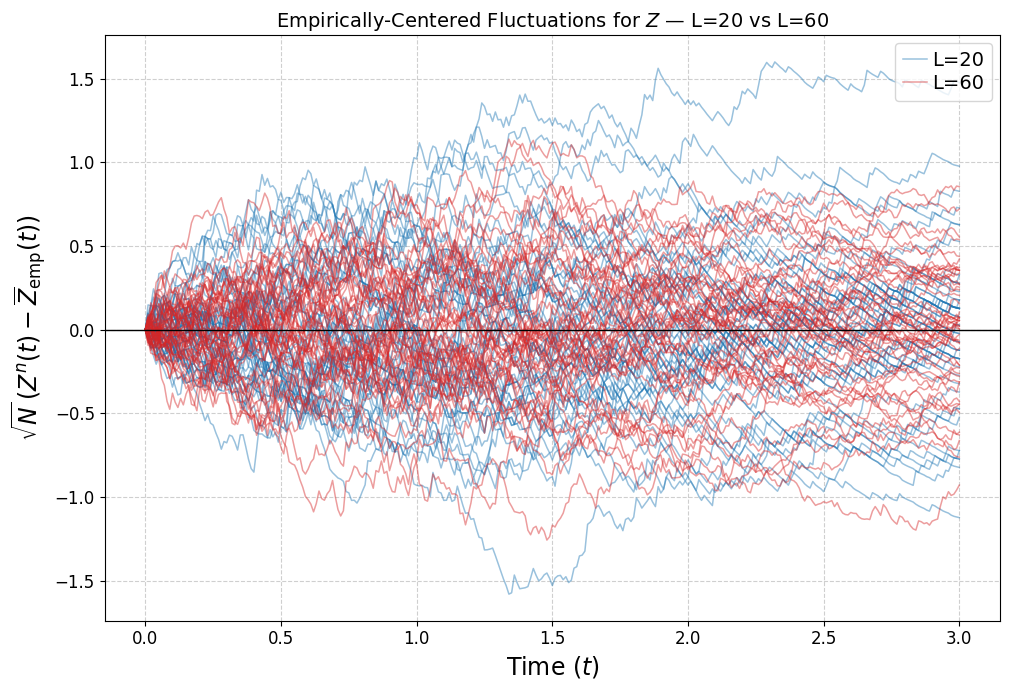}
    \caption{Empirically centered fluctuations $\sqrt{N}\,\big(Z^n(t)-\bar Z_{\mathrm{emp}}(t)\big)$ for $L=50$.}
  \end{subfigure}\hfill
  \begin{subfigure}[b]{0.48\textwidth}
    \includegraphics[trim={0 0 0 0.7cm},clip,width=\linewidth]{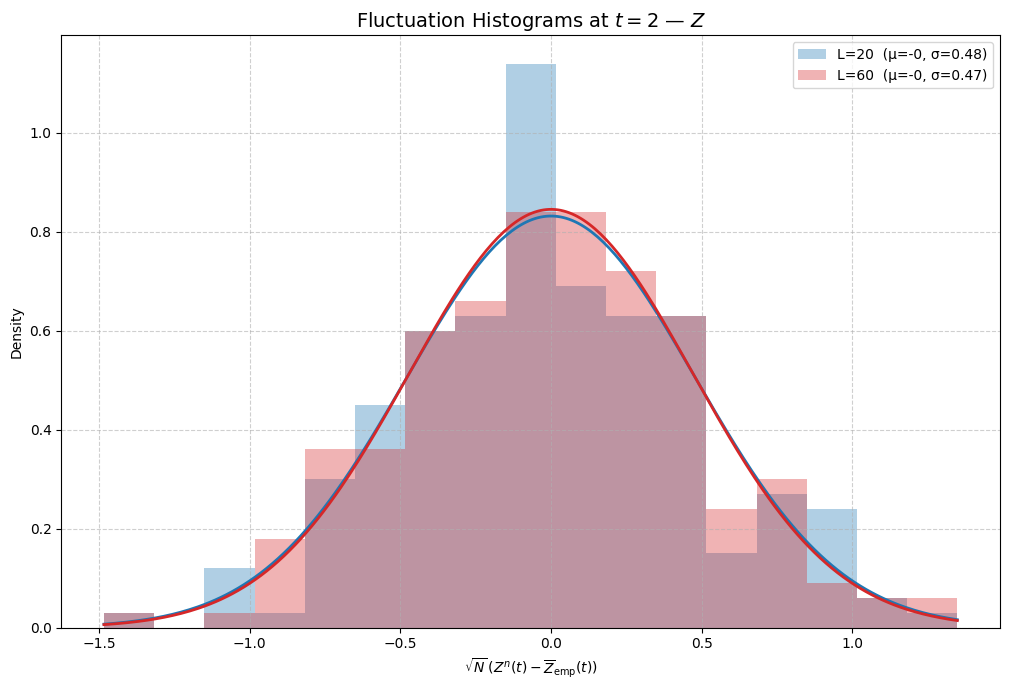}
  \caption{Histogram at $t=2$ of $\sqrt{N}\,\big(Z^n(t)-\bar Z_{\mathrm{emp}}(t)\big)$ with normal fit, based on 100 independent simulations.}

  \end{subfigure}
  \caption[Fluctuations under Weibull]
  {Weibull stifling times with shape $k=2$ and scale $=5$; contact rate $\lambda=0.5$.}
  \label{fig:fluctuations}
\end{figure}

\section{Declaration of generative AI and AI-assisted technologies in the manuscript preparation process}

During the preparation of this work, the authors used the generative AI tools ChatGPT (OpenAI) and Gemini (Google) in order to assist with text revision and code debugging. After using these tools, the authors reviewed and edited the content as needed and take full responsibility for the content of the published article.

\section{Acknowledgements}

This work was supported by FAPESP grants 2023/13453-5 and 2025/02013-0 and CNPq grant 306496/2024-0.

\bibliographystyle{elsarticle-num} 
\bibliography{references}

\begin{description}
  \item[Nancy Lopes Garcia and Denis Araujo Luiz] 
  Instituto de Matemática, Estatística e Computação Científica - Universidade Estadual de Campinas \\
        Campinas, SP, Brazil \\
        E-mails: \texttt{nancyg@ime.unicamp.br}, \texttt{denisalu@ime.unicamp.br}
        
  \item[Daniel Miranda Machado] 
  Centro de Matemática, Computação e Cognição - Universidade Federal do ABC  \\
        Santo André, SP, Brazil \\
        \texttt{daniel.miranda@ufabc.edu.br}
\end{description}

\end{document}